\documentclass[11pt]{article}
\usepackage{microtype}
\usepackage{graphicx}
\usepackage{subfigure}
\usepackage{amsmath,amsbsy,amsfonts,amssymb,amsthm,bm}
\usepackage{algorithm,algorithmic,mathtools}
\usepackage{color,cases,multirow}
\usepackage[margin=1.2in]{geometry} 
\usepackage[colorlinks=True, citecolor=blue]{hyperref}

\usepackage{authblk}

\newtheorem{theorem}{Theorem}
\newtheorem{remark}{Remark}
\newtheorem{proposition}[theorem]{Proposition}
\newtheorem{lemma}{Lemma}
\newtheorem{assumption}{Assumption}

\theoremstyle{definition}

\DeclareMathOperator{\argmin}{argmin}

\DeclareMathOperator{\Tr}{Tr}

\newcommand{\bR}{{R}}
\newcommand{\Bb}{\mathbf{B}}

\newcommand{\RR}{\mathbb{R}}
\renewcommand{\SS}{\mathbb{S}}
\newcommand{\R}{\mathbb{R}}

\newcommand{\EE}{\mathbb{E}}
\newcommand{\ZZ}{\mathbb{Z}}
\newcommand{\E}{\mathbb{E}}
\newcommand{\PP}{\mathbb{P}}

\newcommand{\cH}{\mathcal{H}}
\newcommand{\cX}{\mathcal{X}}
\newcommand{\cN}{\mathcal{N}}
\newcommand{\cP}{\mathcal{P}}
\newcommand{\cF}{\mathcal{F}}

\newcommand{\SE}{\mathcal{S}}
\newcommand{\CR}{\mathcal{R}}

\newcommand{\bvarphi}{\bm{\varphi}}
\newcommand{\bx}{\bm{x}}

\newcommand{\bz}{\bm{z}}

\newcommand{\zb}{\bm{z}}
\newcommand{\bw}{\bm{w}}

\newcommand{\bv}{\bm{v}}
\newcommand{\bu}{\bm{u}}
\newcommand{\ba}{\bm{a}}
\newcommand{\bb}{\bm{b}}

\newcommand{\bp}{\bm{p}}

\newcommand{\bom}{\bm{\omega}}

\newcommand{\wb}{\bm{w}}

\newcommand{\xb}{\bm{x}}

\newcommand{\vb}{\mathbf v}

\newcommand{\Jb}{\mathbf J}

\newcommand{\bpx}{\bm{p}^{\bx}}
\newcommand{\bzx}{\bm{z}^{\bx}}
\newcommand{\rad}{\text{Rad}}

\newcommand{\be}{\begin{equation}}
\newcommand{\ee}{\end{equation}}

\title{Machine Learning from a Continuous Viewpoint I}

\author[1,2]{Weinan E \footnote{Also at Beijing Institute of Big Data Research.}
                   \thanks{\texttt{weinan@math.princeton.edu}}}
\author[2]{Chao Ma \thanks{\texttt{cham@princeton.edu}}}
\author[2]{Lei Wu \thanks{\texttt{leiwu@princeton.edu}}}

\affil[1]{Department of Mathematics, Princeton University}
\affil[2]{Program in Applied and Computational Mathematics, Princeton University}

\date{}

\begin{document}
\maketitle

\begin{abstract}

We present a continuous formulation of machine learning, 
as a problem in the calculus of variations and differential-integral equations,
 in the spirit of classical numerical analysis.
We demonstrate that conventional machine learning models and algorithms,
such as the random feature model, the two-layer neural network model
and the residual neural network model, can all be recovered (in a scaled form) as
particular discretizations of different continuous formulations.
We also present examples of new models, such as the flow-based random feature
model, and new algorithms, such as the smoothed particle method
and spectral method,
that arise naturally from this continuous formulation.
We discuss how the issues of generalization error and implicit
regularization can be studied under this framework.

\end{abstract}
{
  \hypersetup{linkcolor=black}
  \tableofcontents
}

\section{Introduction}

We present a continuous formulation of machine learning. 
As usual this continuous formulation consists of three components:
a representation of functions, a loss functional and a training dynamics.
For  representations of functions, we will discuss the integral transform-based models
and the more advanced flow-based models.
For the loss functional, we give examples that arise in supervised and unsupervised learning, as
well as examples from calculus of variations and partial differential equations (PDEs).
For training dynamics, we divide the unknown parameters into two classes:
conserved and non-conserved. For non-conserved parameters, we use what is known
in the physics literature as the model A dynamics \cite{hohenberg1977theory}, namely gradient flow in the
usual $L^2$ metric.
For conserved parameters, we use what is known as the model B dynamics \cite{hohenberg1977theory}, 
namely the gradient flow in the Wasserstein metric \cite{jordan1998variational}. 

In this framework, machine learning becomes a calculus of variations or PDE-like problem,
and different numerical algorithms can be used to discretize these continuous models. 
In particular, two-layer neural network \cite{cybenko1989approximation,barron1993universal}  and deep residual neural network (ResNet) models \cite{he2016deep,e2019barron}
can be recovered, in a scaled form, when the particle method is  applied to particular versions of the integral transform-based and flow-based models respectively. 
New machine learning models and algorithms can also be constructed using this continuous framework. 
As examples we will discuss a new flow-based random feature model, a new class of transform-based model, 
smoothed particle methods and spectral methods.

In addition to recovering existing machine learning models and constructing new ones, this continuous framework is also useful for the theoretical understanding of machine learning.  
We conjecture that the at least for standard supervised learning, the variational problems that arise from minimizing the population and empirical
risks are nice variational problems in this continuous formulation, though the precise meaning of this remains to be clarified.
Thus, the training models are some versions of the gradient flow of a reasonably nice functional. Hence it is not surprising that stable numerical discretizations of these continuous models perform well. 
In particular, this continuous viewpoint suggests that over-parametrized models should behave better since they 
 give rise to  more accurate discretizations of the continuous gradient flow. 
The behavior of the training algorithm should follow more closely the behavior of the continuous gradient flow. 
This viewpoint also suggests that one should expect trouble for very deep fully connected neural network models (which are 
not ResNets) since they do not have continuum limits. 
In fact, they suffer from numerical instabilities in the form of exploding gradients \cite{hochreiter2001gradient,hanin2018neural}.


This work builds upon previous work.  In particular, various components of this continuous framework have already appeared in the following set of works.
\begin{enumerate}
\item Continuous differential equation formulation of machine learning \cite{
weinan2017proposal, haber2017stable, lu2018beyond, li2017maximum, weinan2019mean,chen2018neural}.
\item The work on the integral representations of shallow neural networks 
\cite{cybenko1989approximation,barron1993universal,murata1996integral,candes1999harmonic,
bach2017breaking,sonoda2018global,leroux07a,sonoda2017neural}.
\item Mean field analysis of (stochastic) gradient descent for two-layer neural networks\cite{mei2018mean,rotskoff2018parameters,sirignano2019mean,chizat2018global}  and multi-layer fully connected networks \cite{araujo2019mean,nguyen2019mean,sirignano2019deep}.
\item The function space work \cite{e2019barron,e2018priori}.
\end{enumerate}
Also related are the work in \cite{rotskoff2019global,bartlett2018representing,avelin2019neural,thorpe2018deep,arbel2019maximum}. 
The work presented here is a natural extension of these ideas.
However, the present paper is the first that systematically explores the continuous viewpoint.

Philosophically the approach advocated here bears a lot of similarity to that of the PDE-type models in image processing, such
as the Mumford-Shah model or the Rudin-Osher-Fatemi model \cite{mumford1989optimal,rudin1992nonlinear}.
There the idea is to first formulate a continuous variational problem that presumably represents the ``first principle'' for the particular
image processing task such as image denoising, and then discretize that continuous problem to obtain a specific algorithm.
This is in contrast to the more traditional approach in image processing in which different types of filters
or algorithms are applied directly to the image, without the need to formulate the underlying mathematical problem first.
This latter approach resembles the current practice in machine learning in which different algorithms are applied
directly to the given dataset.

Despite its unprecedented successes across a wide spectrum of
applications, machine learning still remains to be unsatisfactory as a 
scientific discipline. The main problem is the lack of fundamental guiding principles
for designing machine learning models and algorithms and understanding their performance.
Many of the techniques used in practice are still quite ad hoc, and require heavy parameter tuning.
Often times the performance of these models is quite fragile and sensitive to the choice of the 
hyper-parameters.

The situation is reminiscent of what happened during the 1950's when finite difference
and finite element methods were just invented and used to solve 
PDEs. The performance of the algorithms
were found to be sensitive to the particular discretization schemes and
particular finite element meshes used.
Some seemingly reasonable schemes simply did not run, since they quickly led
to overflow on the computer.
Some schemes performed reasonably well on coarse grids but blowed up
upon refining grids. 

Efforts for building 
a theoretical foundation for these finite difference and finite element methods
did not go smoothly either. For example, 
to explain the overflow
phenomenon often encountered in practice, different stability concepts and criteria
were proposed \cite{forsythe1967finite, richtmyer1959difference}. 
Some were easy to use in practice, but were not robust under perturbations.
One such example was the concept of weak stability \cite{richtmyer1959difference}.
Some were more robust but difficult to use in practice.
It took a while for the numerical analysis community to finally
settle down on the right concepts and criteria \cite{gustafsson1995time}.
But after all the dusts were settled, what emerged was a solid and reasonably simple
picture about the basic concepts and principles behind the designing 
and understanding of these algorithms  \cite{gustafsson1995time, ciarlet2002finite}.

Our current work is very much motivated by the same objective, namely to develop
a reasonably simple and transparent framework for machine learning.
However, there is a key difference between machine learning and classical
numerical analysis: While classical numerical analysis is mainly concerned
with problems in low dimension, machine learning has to face problems in very
high dimensions. In fact, our interest is really on machine learning models
and algorithms that can overcome ``the curse of dimensionality''.
While the exact meaning of this terminology requires qualification, the dimensionality
issue is certainly among the most important considerations in developing
machine learning models and machine learning theory today.
In fact, one can roughly divide all machine learning models into two
categories:  The ones that do suffer from the curse of dimensionality
and the ones that do not.  We refer to \cite{e2019notice} for a discussion on this.

One important class of algorithms that do not suffer from the curse
of dimensionality problem are the Monte Carlo algorithms for numerical
integration. In this case, one can establish simple dimension-independent
error rates.  In contrast, grid-based numerical integration methods such
as  Simpson's rule do not share this property. Indeed, their performance
deteriorates rapidly as the dimensionality goes up.

This example has some important consequences on the formulation that we will present:
\begin{enumerate}
\item We will focus on ways of representing functions as expectations,
since there are algorithms for computing expectations with dimension-independent
error rates.
\item For the same reason, particle methods stand out for the training dynamics
since they are the analog of Monte Carlo methods for dynamic problems.
\end{enumerate}

\begin{remark}
There are important aspects of machine learning that cannot be easily formulated at the continuous level. One example is the stochastic gradient descent algorithm.
\end{remark}

\paragraph{Notations.} For any function $f: \RR^m\mapsto\RR^n$, let  $\nabla f = (\frac{\partial f_i}{\partial x_j})_{i,j}\in\RR^{n\times m}$ and $\nabla^T f = (\nabla f)^T$. We use $X\lesssim Y$ to mean $X\leq CY$ for some absolute constant $C$. For any $\bx\in\RR^d$, let $\tilde{\bx}=(\bx^T,1)^T\in\RR^{d+1}$. 
Let $\Omega$ be a subset of $\RR^d$, and denote by $\cP(\Omega)$ the space of probability measures. Define $\cP_2(\Omega)=\{\,\mu\in \cP(\Omega): \int \|\bx\|_2^2 d\mu(\bx)<\infty\}$.
We will also follow the convention in probability theory for denoting function dependence, e.g.
$\rho_t$ means the value of $\rho$ at time $t$.

\section{Representations of functions}

We are mainly interested in representations that are potentially effective in high dimensions.
Therefore we will focus on the ones that can be expressed as expectations.
As an example, instead of the Fourier representation:
\begin{equation}
\label{Fourier-1}
f(\bx) = \int_{\R^d} a(\bom) e^{i (\bom, \bx)} d \bom,
\end{equation}
we will consider
\begin{equation}
\label{Fourier-2}
f(\bx) = \int_{\R^d} a(\bom) e^{i (\bom\bx)}  \pi(d\bom) =
\E_{\bom \sim \pi} a(\bom) e^{i (\bom, \bx)}
\end{equation}
where $\pi$ is a probability measure on $\R^d$. The reason that we prefer \eqref{Fourier-2} over
\eqref{Fourier-1} is as follows. The discrete analog of \eqref{Fourier-1} is
\begin{equation}
\label{Fourier-1.1}
f_m(\bx) = \frac 1 m \sum_j a(\bom_j) e^{i (\bom_j, \bx)} 
\end{equation}
 where the sum is performed on a regular grid $\{\bom_j\}_{j=1}^m$ in the Fourier space. It is well-known that this kind of grid-based  approximations
satisfies
\begin{equation}
\label{Fourier-1.2}
f - f_m   \sim C(f) m^{-\alpha/d}
\end{equation}
where $C(f)$ and $\alpha$ are fixed quantities depending on $f$. The appearance of $1/d$ in the exponent of $m$ signals the curse of
dimensionality.  In contrast, for \eqref{Fourier-2},
by independently sample $\{\bom_j\}_{j=1}^m$ from $\pi$, we obtain an approximation to $f$ with a dimension-independent
error rate:
$$
\E|f(\bx) - \frac 1m \sum_{j=1}^m a(\bom_j) e^{i (\bom_j, \bx)}|^2 = 
\frac{\mbox{var}(f)}{m} 
$$
where
$$
\mbox{var}(f) = \E_{\bom \sim \pi} |a(\bom)|^2 - f(\bx)^2
$$

Equation \eqref{Fourier-2} can also be written as:
\begin{equation}
\label{Fourier-3}
f(\bx) = \int_{\R^d} a e^{i (\bom, \bx)}  \rho(da, d\bom) =
\EE_{(a, \bom) \sim \rho} a e^{i (\bom, \bx)}
\end{equation}
where $   \rho(da, d\bom)  = \delta(a - a(\bom))da \pi(d \bom) $.

From an algorithmic viewpoint, \eqref{Fourier-1} is typically associated with non-adaptive discretizations such as
the spectral method \cite{gottlieb1977numerical} or the ridglets and curvelets used in signal processing \cite{candes1999ridgelets,candes1999harmonic, murata1996integral}.
We will see later that the  forms \eqref{Fourier-2} and \eqref{Fourier-3} are closely 
associated with  the random feature model
and the two-layer neural network model. In fact one can write 
\[
\frac 1m \sum_{j=1}^m a(\bom_j) e^{i (\bom_j, \bx)} = \frac 1m \sum_{j=1}^m a_j \sigma (\bom_j, \bx)
\]
where $\sigma$ is defined by $\sigma (z) = e^{iz} $.  This is a two-layer neural network with activation function $\sigma$.

\subsection{Integral-transform based representation}\label{ssec:transform}

\paragraph*{Generalized ridgelet transforms}

The original ridgelet transform representation is as follows \cite{candes1999ridgelets, murata1996integral}
$$
f(\bx) = \int_{\R^{d+1}} a(\bw) \sigma(\bw^T \tilde{\bx}) d \bw
$$
where $\sigma$ is a nonlinear scalar function, the analog of the activation
function in neural networks. { Here we used $\tilde{\bx}=(\bx^T,1)^T\in\RR^{d+1}$ to include the bias term. 
To simplify the notation, in the rest of this paper we will write $\bx$ instead of $\tilde{\bx}$ when it is clear from the context that the bias term should be present.
}
Motivated by the discussions above, we define the generalized ridgelet transform by 
\be
\label{ridge-1}
f(\bx) = \int_{\R^{d+1}} a(\bw) \sigma(\bw^T \bx) \pi(d \bw)
= \E_{\bw \sim \pi} a(\bw) \sigma(\bw^T \bx)
\ee
This representation is better suited for high dimensional situations.  More generally, one can also use:
\be
\label{ridge-2}
f(\bx) = \int_{\R^{d+2}} a \sigma(\bw^T \bx) \rho(da, d \bw)
= \E_{(a, \bw) \sim \rho} a \sigma(\bw^T \bx)
\ee

\paragraph*{High co-dimensional representation}

Ridgelet transforms express function in terms of superpositions of ridge-like
structures which are co-dimension one objects.  One can also refine this 
representation, using structures of high co-dimension.  For example, the
following representation uses co-dimension 2 objects:
\begin{align}\label{eqn: co-ridge}
\nonumber f(\bx) &= \int_{\R^d \times \R^d} a(\bw_1, \bw_2) \sigma_1(\bw_1^T \bx) 
\sigma_2(\bw_2^T \bx)\pi(d \bw_1, d\bw_2)\\
&= \E_{\bw \sim \pi} a(\bw_1, \bw_2)\sigma_1(\bw_1^T \bx) \sigma_2(\bw_2^T\bx)
\end{align}
where $\bw=(\bw_1, 
\bw_2)$, $\sigma_1$ and $\sigma_2$ are two nonlinear scalar functions.

More generally, we can consider functions of the form
\be\label{eqn: transform-rep}
f(\bx) 
= \E_{\bw \sim \rho} [\varphi(\bx;\bw)],
\ee
where $\bw\in\Omega$ and $\rho\in\cP(\Omega)$.
Note that 
\eqref{ridge-1} and \eqref{eqn: co-ridge} are both special cases of the  representation
above.

\paragraph*{Compositional structure}

The representations discussed above correspond to neural network models with one hidden layer.
It  can be straightforwardly extended to include more hidden layers using a compositional structure.
An example with two hidden layers is given by:
\begin{equation}
f(\bx) = \int_{\R^{d_1}} a_1(\bw_1) \sigma(\bw_1^T \zb) \pi_1(d \bw_1)
= \E_{\omega_1 \sim \pi_1} a_1(\bw_1) \sigma(\bw_1^T \zb)
\end{equation}
\begin{equation}
\zb = \int_{\R^{d_2}} \ba_2(\bw_2) \sigma(\bw_2^T \tilde{\bx}) \pi_2(d \bw_2)
= \E_{\bw_2 \sim \pi_2} \ba_2(\bw_2) \sigma(\bw_2^T \tilde{\bx})
\end{equation}
where $d_2 = d+1$, $\pi_1, \pi_2$ are probability measures on $\R^{d_1}$ and $\R^{d_2}$ respectively,
$a_1: \R^{d_1} \rightarrow \R^1$, $\ba_2: \R^{d_2} \rightarrow \R^{d_1}$.

Another way to construct compositional structures is as follows:
\begin{equation}
f(\bx) = \int_{\R^{d_1}} a_1 \sigma(\bw_1^T \zb) \rho_1(da_1, d \bw_1)
= \E_{(a_1, \bw_1) \sim \rho_1} a_1 \sigma(\bw_1^T \zb)
\end{equation}
\begin{equation}
\zb = \int_{\R^{d_2}} a_2 \sigma(\bw_2^T \tilde{\bx}) \rho_2(d \bw_2)
= \E_{(a_2, \bw_2) \sim \rho_2} a_2 \sigma(\bw_2^T \tilde{\bx})
\end{equation}
{where $d_2 = d+1$}, $\rho_1, \rho_2$ are probability measures on $\R^1 \times \R^{d_1}$ and $\R^{d_1} \times \R^{d_2}$ respectively,

\subsection{Flow-based representation}

In the flow-based representation, the trial functions are generated by the
flow map of a (continuous) dynamical system:
$$ \frac{d \bz}{d\tau} = g(\tau, \bz),\, \bz(0) = \tilde{\bx} = (\bx^T, 1)^T
$$
The flow-map at time $1$ is defined as the map: $ \bx \rightarrow \bz(1) $.
More generally, one can allow a change of dimension between $\bx $ and $\bz$:
$$ \frac{d \bz}{d\tau} = g(\tau, \bz),\, \bz(0) = V \tilde{\bx}
$$
where $V \in \R^{D \times (d+1))}$ is a $D\times (d+1)$ matrix with rank $d+1$.
Scalar functions can be obtained by contracting this map with a vector:
$$ 
f(\bx) = \bm{\alpha}^T \bz(1), \, \bm{\alpha} \in \R^D
$$

The set of functions that can be generated this way depend on how we choose
$g$.  One natural way  is to use the representation discussed
above, e.g. 
$$
g(\tau, \bz) 
= \E_{\bw \sim \pi_\tau} \ba(\bw, \tau) \sigma(\bw^T \bz)
$$
Here $(\pi_\tau)_{\tau\in[0,1]}$ is a family of probability distributions parametrized by $\tau$.  This gives us the flow:
\be
\label{flow-2}
 \frac{d \bz}{d\tau} 
= \E_{\bw \sim \pi_\tau} \ba(\bw, \tau) \sigma(\bw^T \bz)
\ee
As before, we can also use the model:
\be
\label{flow-3}
 \frac{d \bz}{d\tau} =  \int_{\R^D} \ba \sigma(\bw^T \bz) \rho_\tau(d\ba, d \bw)
= \E_{(\ba, \bw) \sim \rho_\tau}  \ba  \sigma(\bw^T \bz)
\ee
where $(\rho_\tau)_{\tau\in[0,1]}$ is a family of probability distributions on $\R^D \times \R^D$.
More generally, we can consider the following model
\be \label{flow-4}
     \frac{d \bz}{d\tau} = \EE_{\bw\sim\rho_\tau}[\bvarphi(\bz;\bw)],
\ee 
where $\bvarphi(\bz;\bw)\in\RR^D$.

If we compare the models in \eqref{flow-2} and \eqref{flow-3} with the model proposed originally in \cite{weinan2017proposal}:
\be
\label{flow-5}
 \frac{d \bz}{d\tau} =   \bu(\tau) \sigma (\bw(\tau) ^T \bz) 
\ee
we see that \eqref{flow-5} is the special case of \eqref{flow-3} with $\rho_\tau = \delta(\ba-\bu(\tau))\delta(\bw-\bw(\tau))$.

However, as we learn from the work of \cite{e2019barron}, the more general representation in \eqref{flow-2} and \eqref{flow-3} is needed in order to capture the continuum limit of residual neural networks.

\section{The optimization problem}
The next step is to formulate the loss function that will be used in order
to turn the problem into an optimization problem. 
In the continuous setting, these optimization problems are calculus of variations problems.
Here we will discuss four examples of machine learning tasks.

{
In the following we will use $\theta$ to denote generically the set of parameters that occur in the representation.
For example, for \eqref{ridge-1}, we have $\theta = (a(\cdot), \pi(\cdot))$.
}

\subsection{Supervised learning}

In supervised learning, our objective is to find the best approximation of some target function $f^*$ that minimizes the so-called
population risk:
\be
\label{supervised-1}
\CR (\theta) = \int_{\R^d} (f(\bx;\theta) - f^*(\bx) )^2 \mu (d \bx)
\ee
Here $\mu$ is a probability distribution.  \eqref{supervised-1} is the $L^2$ loss function. Obviously one can also define
other loss functions by replacing the square function by some other convex functions with the global minimum at the origin.
The general form of loss function is given by:
\begin{align}
\label{supervised-1.1}
    \mathcal{R}(\theta) = \EE_{\bx}[\ell(f(\bx;\theta),f^*(\bx))].
\end{align}
Here $\ell$ is a convex function, $y$ is the (possibly noisy) label associated with $\bx$.

In reality,  we are only given partial information about $f^*$ and $\mu$ through a finite sample: 
 $S=\{(\xb_i, y_i)\}_{i=1}^n$ where $y_i=f^*(\bx_i)$ is the label for $\xb_i$.
Therefore in practice we have to work instead with
the ``empirical risk'':
\be
\label{supervised-2}
\hat{\CR}_n (\theta) = \frac 1n \sum_{i=1}^n  \ell(f(\bx_i; \theta),f^*(\bx_i)).
\ee

\subsection{Dimension reduction}

Dimension reduction is an important problem in unsupervised learning.
Here we are given a  dataset $\SE=\{\xb_i \}_{i=1}^n \in \R^d$  where $\{\xb_i \}$ are sampled from an underlying probability distribution $\mu$.
Our assumption is that $\mu$ is concentrated on a lower dimensional set in $\R^d$ and we would like to find a set of coordinates (functions of
$\xb$) that characterize that low dimensional set.  Let us assume  that the dimension of the low dimensional set is $D$ and is known to us.
Define two functions: an encoder $f: \R^d \rightarrow \R^{D}$ and a decoder $g: \R^{D} \rightarrow \R^{{d}}$. 
The encoder is a compression map and the decoder is a reconstruction map.
Our objective is to minimize the reconstruction error:
{
\be
\CR (\theta_1,\theta_2) = \int_{\R^d} (\bx - g(f(\bx; \theta_1); \theta_2))^2 \mu (d \bx)
\ee
}
This is the analog of the population risk. 
Again in practice, one has to work with the empirical risk, defined by:
{
\be
\hat{\CR}_n (\theta_1,\theta_2) = \frac 1 n \sum_{i=1}^n  (\bx_i - g(f(\bx_i; \theta_1);\theta_2))^2
\ee
}


\subsection{Calculus of variations}

A particularly important problem is the ground state of a quantum system \cite{carleo2017solving,han2019solving,Pfau2019AbInitioSO}. Let $\mathcal{H}$ be the  Hamiltonian operator of the
quantum system, say on $\R^d$.  The ground state  is the minimizer of the energy:
\be
\mathcal{I}(\varphi) = \frac {\int_{\R^d} \varphi^*(\bx) {\cH} \varphi(\bx) d \bx} {\int_{\R^d} |\varphi(\bx)|^2 d \bx}
\ee
where $\varphi^*$ is the complex conjugate of $\varphi$.
This energy can also be rewritten as
\be
\mathcal{I}(\varphi) = \E_{\bx \sim \mu_{\varphi}}  \frac {\varphi^*(\bx) {\cH}\varphi(\bx) }{|\varphi(\bx)|^2}
\ee
where $\mu_{\varphi}$ is the probability distribution defined by:
\be 
\mu_{\varphi}(d \bx) = \frac 1Z |\varphi(\bx)|^2 d \bx, \quad Z = \int_{\R^d} | \varphi(\bx)|^2 d \bx
\ee
$\mathcal{I}$ serves as the analog of the population risk.

In these problems, one typically attempts to compute $\mathcal{I}$ accurately by producing sufficient number of samples from the distribution $\mu_{\varphi}$.
This means that one attempts to work directly with the population risk in these problems.
In practice, however, there is an issue that the errors in the approximation of $\varphi$ may interact with the errors in the sampling.
This issue has not been systematically investigated yet.

\subsection{Nonlinear parabolic PDEs}

An important application of machine learning is the numerical solution of high dimensional PDEs
\cite{han2016deep, carleo2017solving, han2018solving, weinan2017deep, sirignano2018dgm, weinan2018deep, khoo2019solving, han2019solving, Pfau2019AbInitioSO}.
Formulating these PDEs as variational problems is an important step in formulating machine learning based algorithms.
In principle, one can always use the ``least square'' approach, as was done in \cite{carleo2017solving, sirignano2018dgm}.
But better performance can be achieved if more sophisticated formulations are used.

Consider the nonlinear parabolic PDE:
\begin{align}\label{eq:PDE}
          \frac{ \partial u}{ \partial t } ( t, x )
          + \frac{1}{2} \! \Tr\!\Big( \Sigma\Sigma^{\operatorname{T}}(t,x)(\mbox{Hess}_x &u) ( t, x ) \Big)
          +\nabla u( t, x )\cdot \mu( t, x ) \\
         & + h\big( t, x, u(t,x), \Sigma^{\operatorname{T}}( t, x ) \nabla u( t, x ) \big) = 0.
\end{align}
with the terminal condition $u(T,x) = g(x)$.  Among other things, this kinds of PDEs arise in option pricing with 
default risk or other nonlinear effects taken into account.

It can be shown that this PDE problem is equivalent to the following variational problem \cite{han2018solving, weinan2017deep}
\begin{align*}
&\inf_{Y_0,\{Z_t\}_{0\le t \le T}} \E |g(X_T) - Y_T|^2, \\
&s.t.\quad X_t = \xi + \int_{0}^{t}\mu(s,X_s)\, \,ds + \int_{0}^{t}\Sigma(s,X_s)\, dW_s, \\
&\hphantom{s.t.}\quad Y_t = Y_0 - \int_{0}^{t}h(s,X_s,Y_s,Z_s)\,  ds + \int_{0}^{t}(Z_s)^{\operatorname{T}}\, dW_s.
\end{align*}
The constraints are backward stochastic differential equations (BSDE) \cite{pardoux1992backward}.
This was the starting point of the ``Deep BSDE method'' proposed in \cite{han2018solving, weinan2017deep}.

Analyzing these variational problems is a major task in the mathematical theory of machine learning.

From now on we will focus on the supervised learning problem.

\section{Gradient flows}

The third component in machine learning is an algorithm for solving the optimization problem.
In this section, we will discuss various gradient flow
dynamics for the population or empirical risk. {For simplicity  we focus on the following loss functional
\be\label{eqn: loss-functional}
    \CR(\theta) = \EE_{\bx}[\ell(f(\bx;\theta),f^*(\bx))],
\ee 
where  $\ell(y_1,y_2)=(y_1-y_2)^2/2$.
}

\subsection{Conservative and non-conservative gradient flows}

We first discuss  gradient flows using a physics language \cite{hohenberg1977theory}.
The loss functions or functionals defined above serve as the ``free energy'' of the problem.

To begin with, we need to distinguish
conserved and non-conserved ``order parameters''. 
The coefficient $a$ in \eqref{ridge-1} is non-conserved.
The probability distributions $\pi$ or $\rho$ are obviously conserved. 

First, let us examine the situation with the representation \eqref{ridge-1}.
Let $I=I(a, \pi)$ be the loss functional. 
Denote by $\frac{\delta I}{\delta a}$ and $\frac{\delta I}{\delta \pi}$ the
formal variational derivative of $I$ with respect to $a$ and $\pi$ respectively,
under the standard $L^2$ metric.
The gradient flow for $a$ is simply given by 
\be\label{eqn: A-dynamics}
\frac{\partial a}{\partial t} = - \frac{\delta I}{\delta a}
\ee
In the physics literature, this is known as the ``model A'' dynamics \cite{hohenberg1977theory}.

The gradient flow for $\pi$ is given by a continuity equation:
\be \label{eqn: B-dynamics}
\frac{\partial \pi}{\partial t} + \nabla \cdot {\bf J} = 0
\ee
where the current $\Jb$ is given by:
$$
{\bf J}= \pi \vb, \, \vb= - \nabla V
$$
$$
V= \frac{\delta I}{\delta \pi}.
$$
This is known as the ``model B'' dynamics \cite{hohenberg1977theory} and $V$
is known as the ``chemical potential''.

\begin{remark}
It is well-known that the model B dynamics is also the
gradient flow under the 2-Wasserstein metric \cite{jordan1998variational,villani2008optimal}.
\end{remark}

{
For flow-based models, the parameters $a$ and $\pi$ are themselves
one-parameter families of coefficients or probability distributions respectively:
$a = (a_\tau)_{\tau \in [0, 1]}, \pi = (\pi_\tau)_{\tau \in [0, 1]}$.
Given a functional $I, I = I((a_\tau), (\pi_\tau))$,  
a natural extension of the gradient flow  to
$(a_\tau), (\pi_\tau)$ is given by:
$$
\frac{\partial a_\tau}{\partial t} = - \frac{\delta I}{\delta a_\tau}
$$
$$
\frac{\partial \pi_\tau}{\partial t} + \nabla \cdot {\bf J_\tau} = 0,
$$
where 
\[
{\mathbf J}_\tau= -\pi_\tau \nabla \frac{\delta I}{\delta \pi_\tau}.
\]
Note that the varitional derivatives $\delta I/\delta a_\tau$ and $\delta I/\delta \pi_\tau$ 
appeared above are not well-defined, since $I$ is the integral of the
influences of $\ba_\tau$ and $\pi_\tau$ from $\tau=0$ to $\tau=1$. So strictly speaking,
these derivatives are infinitesimal quantities. 
In section \ref{sec: flow-rand-feat} and \ref{sec: flow-neural-net}, we will provide 
rigorous forms of these equations.
}

\paragraph*{Example 1:  An example of the conserved parameter}
Consider representations of the form
$$f(\bx) = \int \varphi(\xb, \wb) \pi(d \wb).
$$
The {chemical potential} for this functional is given by
$$
V (\wb) = \frac{\delta \mathcal{R}}{\delta \pi}(\wb) = \EE_{\bx} [(f(\bx) - f^*(\bx)) \varphi(\bx, \wb)]
=\int K(\wb, \tilde{\wb}) \pi(d \tilde{\wb}) - \tilde{f}(\wb)
$$
where
\be\label{eqn: k-f-def}
\begin{aligned}
K(\wb, \tilde{\wb}) &= \EE_{\bx} [\varphi(\xb, \wb) \varphi(\bx, \tilde{\wb})]\\
\tilde{f}(\wb) &= \EE_{\bx}[f^*(\bx)\varphi(\bx, \wb)]
\end{aligned}
\ee

The model B gradient flow in this case is given by
$$
\partial_{t} \pi + \nabla (\pi \nabla V) = 0
$$

This is nothing but the ``mean field'' limit of the  gradient descent dynamics for two-layer neural networks 
\cite{mei2018mean,rotskoff2018parameters,sirignano2019mean,chizat2018global}.

\paragraph*{Example 2: An example of non-conserved parameter}

Consider now the representation
$$f(\bx) = \int a(\wb) \sigma(\wb^T {\xb}) \pi(d \wb)
$$
with  $\pi$ being fixed.
The variational derivative of \eqref{eqn: loss-functional} with respect to $L^2(\pi)$  is given by
$$
\frac{\delta \mathcal{R}}{\delta a}(\wb) = \EE_{\bx}[(f(\bx) - f^*(\bx)) \sigma(\wb^T {\bx})]
=\int K(\wb, \tilde{\wb}) \pi(d \tilde{\wb}) - \tilde{f}(\wb)
$$
where $K$ and $\tilde{f}$ are defined  as in \eqref{eqn: k-f-def}.
The gradient flow for $a$ is now given by:
$$
\partial_{t} a(\bw,t) = - \frac{\delta \mathcal{R}}{\delta a}(\wb)
$$

This is the continuous version of the gradient flow for random feature models.

{
\subsection{Pontryagin's maximum principle for flow-based models}
Consider a general flow-based model $f(\bz;\theta)=\bm{1}^T\bz_1^{\bx}$ with $\bz_1^{\bx}$ given by the following ODE,
\begin{equation}\label{eqn: flow-based}
    \frac{d\bz^{\bx}_\tau}{d\tau} = g(\bz^{\bx}_\tau;\theta_\tau),\qquad \bz^{\bx}_0= V\tilde{\bx}.
\end{equation}
Minimizing the risk subject to the dynamics defined by \eqref{eqn: flow-based} is a control problem
where  $\{\bz^{\bx}_\tau \}$ are the states and the parameters $\theta=\{\theta_\tau \}$ serve as the control.
Naturally we will borrow concepts from control theory.
To simplify the statement of the results, we define the following quantity: 
\begin{equation}\label{eqn: Hamiltonian-general}
        H(\bz,\bp,\theta) = \bp^Tg(\bz;\theta).
\end{equation}
Following the convention in control theory, we call $\bz,\bp, H$ the state, co-state and the Hamiltonian, respectively.

\textbf{The Pontryagin's maximum principle} (PMP).  This is a  necessary condition  for the optimal solutions of control problem \cite{boltyanskii1960pontryagin}. In the current case,  let $\theta$ be a global minimum of the risk functional. Then it must satisfy 
\begin{equation}
\theta_\tau = \text{argmin}_{\mu} \EE_{\bx}H(\bz_\tau^{\bx}, \bp_\tau^{\bx},\mu),
\end{equation}
where for each $\bx$, $(\bz^{\bx}_\tau, \bp^{\bx}_\tau)$ is given by  the Hamiltonian dynamics:
\begin{align}
\frac{d\bz^{\bx}_\tau}{d\tau} &=  \nabla_{\bp} H (\bz_\tau^{\bx}, \bp_\tau^{\bx}, \theta_\tau) = g(\bz_\tau^{\bx};\theta_\tau)\\
\frac{d\bp^{\bx}_\tau}{d\tau} &= - \nabla_{\bz} H(\bz_\tau^{\bx}, \bp_\tau^{\bx}, \theta_\tau) = - \nabla_{\bz} g(\bz^{\bx}_\tau, \theta_\tau) \bp_\tau^{\bx},
\end{align}
with the boundary condition $\bz^{\bx}_0 = V \tilde{\bx}, \bp^{\bx}_1 = \bm{1}\ell'(\bm{1}^T\bz_1^{\bx}, f^*(\bx))$.

Note that the dynamics of the state $\bz^{\bx}_\tau$ is forward in time from $\tau=0$ to $\tau=1$, whereas the dynamics of the co-state $\bp^{\bx}_\tau$ is backward in time from $\tau=1$ to $\tau=0$. We refer the reader to \cite{weinan2019mean} for the proof and more discussions.
}

\subsection{Flow-based random feature model}\label{sec: flow-rand-feat}

First a remark about notation. We will use $t$ to denote
 the time for the gradient flow, and  $\tau$ to denote the ``time'' used to define the flow-based models.
 
Consider the following  model 
\begin{equation}\label{eqn: compositional-random-feature}
\begin{aligned}
\bz_0^{\bx} &= V \tilde{\bx} \\
\frac{d\bz_{\tau}^{\bx} }{d\tau}&= \EE_{\bw\sim \pi_\tau}[\ba_\tau(\bw)\varphi(\bz_{\tau}^{\bx},\bw)],\\
f(\bx;\ba) &= \bm{1}^T \bz_1^{\bx},
\end{aligned}
\end{equation}
where { $V\in\RR^{D\times (d+1)}, \bm{1}  = (1, 1, \cdots, 1)^T$, 
$\varphi(\cdot,\cdot): \RR^D\times \Omega\mapsto \RR$ are the features. 
 $V$ is fixed and  $\mbox{rank}(V)=d+1$. 
}We will consider the case when $(\pi_\tau)_{\tau\in [0,1]}$ is pre-fixed. We call this the ``flow-based random feature model''.
 For $\ba$, we define its $L^2$ norm by 
\[
    \|\ba\|_{L^2}^2 = \int_0^1 \int \|\ba_\tau(\bw)\|_2^2 \pi_\tau(d \bw)d\tau.
\]
{
In this case, the Hamiltonian is given by 
\begin{equation}
H(\bz,\bp,\ba) = \EE_{\bw\sim\pi_\tau}[\bp^T\ba(\bw)\varphi(\bz,\bw)].
\end{equation}
}

\begin{proposition}
{ For the loss functional \eqref{eqn: loss-functional},} we have
\[
\frac{\delta \CR}{\delta \ba} = \EE_{\bx}[\bpx_\tau \varphi(\bzx_\tau,\bw)].
\]
where $\bz_{\tau}^{\bx}, \bp_{\tau}^{\bx}$ satisfies the following Hamiltonian dynamics,
\begin{equation}\label{eqn: forward-back-fbrf}
\begin{aligned}
\frac{d \bz_{\tau}^{\bx}}{d \tau} &= \nabla_{p} H = \EE_{\bw\sim\pi_\tau}[\ba_\tau(\bw)\varphi(z_{\tau}^{\bx}, \bw)] \\
\frac{d \bp_{\tau}^{\bx}}{d\tau} &= - \nabla_{\bz} H = - \EE_{\bw\sim\pi_{\tau}} [\ba_\tau(\bw)^T\nabla_{\bz} \varphi(\bz_\tau^{\bx},\bw)] \bp_{\tau}^{\bx},
\end{aligned}
\end{equation}
with the boundary conditions $\bz_{0}^{\bx} = V \tilde{\bx},  \bp_1^{\bx} = \bm{1} \ell'(\bm{1}^T\bzx_1, f^*(\bx))$.
\end{proposition}

\begin{proof}
Let $\bz, \tilde{\bz}$ denote the original and  perturbed states generated by $\ba$  and $\ba+\varepsilon\tilde{\ba}$, respectively.  Then
we have
\begin{align}\label{eqn: xxxx}
 \nonumber   \CR(\ba+\varepsilon \tilde{\ba}) - \CR(\ba) &= \EE_{\bx}[\ell(\bm{1}^T\tilde{\bz}_1^{\bx},f^*(\bx))] - \EE_{\bx}[\ell(\bm{1}^T\bzx_1, f^*(\bx))]\\
    &=  \EE_{\bx}[\ell'(\bm{1}^T\bzx_1, f^*(\bx))\bm{1}^T(\tilde{\bz}_1^{\bx}-\bz_1^{\bx})] + o(\EE_{\bx}[\|\tilde{\bz}_1^{\bx}-\bzx_1\|])
\end{align}
We want to estimate $\tilde{\bz}_1^{\bx} - \bz_1^{\bx}$. We know that $\bz_0^{\bx}=\tilde{\bz}_0^{\bx}$ and 
\begin{align}
    \frac{d\bz_\tau^{\bx}}{d\tau} &= \EE_{\bw\sim\pi_\tau}[\ba_\tau(\bw)\varphi(\bz_\tau^{\bx},\bw)]\\
    \frac{d\tilde{\bz}_\tau^{\bx}}{d\tau} &= \EE_{\bw\sim\pi_\tau}[(\ba_\tau(\bw)+\varepsilon \tilde{\ba}_\tau(\bw))\varphi(\tilde{\bz}_\tau^{\bx},\bw)].
\end{align}
Let $\Delta_\tau^{\bx} = \tilde{\bz}_\tau^{\bx}  - \bz_\tau^{\bx}$, then $\Delta_0^{\bx}=0$ and 
\begin{align}
\frac{d\Delta_\tau^{\bx}}{d\tau} &=  \EE_{\bw\sim\pi_\tau}[\ba_\tau(\bw)(\varphi(\tilde{\bz}^{\bx}_\tau,\bw)-\varphi(\bz^{\bx}_\tau,\bw))] + 
                \varepsilon \EE_{\bw\sim \pi_\tau}[\tilde{\ba}_\tau(\bw)\varphi(\tilde{\bz}_\tau^{\bx},\bw)]\\
                &= \EE_{\bw\sim\pi_\tau}[\ba_\tau(\bw)\nabla_{\bz}^T \varphi(\bz^{\bx}_\tau,\bw)]\Delta_\tau^{\bx} + \varepsilon \EE_{\bw\sim\pi_\tau}[\tilde{\ba}_\tau(w)\varphi(\bz_\tau^{\bx},\bw)] + o(\varepsilon).
\end{align}
Let $Q_\tau = \EE_{\bw\sim\pi_\tau}[\ba_\tau(\bw)\nabla^T_{\bz} \varphi(\bz^x_\tau,\bw)]$,  we have 
\[
    \Delta_1^{\bx} = \varepsilon \int_0^1 e^{\int_\tau^1 Q_s ds} \EE_{\bw\sim\pi_\tau}[\tilde{\ba}(\bw)\varphi(\bz_\tau^{\bx},\bw)] d\tau + o(\varepsilon).
\]
Plugging this into Eqn. \eqref{eqn: xxxx} gives us 
\begin{align}
    \lim_{\varepsilon\to 0}\frac{\CR(\ba+\varepsilon \tilde{\ba})-  \CR(\ba)}{\varepsilon} & = \EE_{\bx}[\ell'(\bm{1}^T\bz_1^{\bx}, f^*(\bx))\langle \bm{1}, \int_0^1 e^{\int_\tau^1 Q_s ds} \EE_{\bw\sim\pi_\tau}[\tilde{\ba}(\bw)\varphi(\bz_\tau^{\bx},\bw)] d\tau \rangle ] \\
    &= \int_0^1 \EE_{\bw\sim\pi_\tau}\left\langle \EE_{\bx}[ e^{\int_\tau^1 Q^T_s ds}\ell'(\bm{1}^T\bz_1^{\bx},f^*(\bx))\bm{1} \varphi(\bz_\tau^x,\bw)],  \tilde{\ba}(\bw)\right\rangle   d\tau  \\
    &= \int_0^1 \EE_{\bw\sim\pi_\tau}\left\langle \EE_{\bx}[\bp^{\bx}_\tau \varphi(\bz_\tau^{\bx},\bw)],  \tilde{\ba}(\bw)\right\rangle   d\tau ,
\end{align}
where we defined the co-state 
\[
\bp_\tau^{\bx} = e^{\int_\tau^1 Q_s^T ds }\bm{1}\ell'(\bm{1}^T\bz_1^{\bx},f^*(\bx)).
\]
The variational derivative of the loss functional is given by 
\[
\frac{\delta \CR}{\delta \ba} = \EE_{\bx}[\bpx_\tau \varphi(\bzx_\tau,\bw)].
\]
Obviously, the co-state satisfies the following backward ODE,
\begin{align}
    \bp_1^{\bx} &= \ell'(\bm{1}^T\bz_1^{\bx},f^*(\bx)) \bm{1} \\
    \frac{d\bp^{\bx}_\tau}{d\tau} &= - Q_\tau^T \bp_\tau^{\bx} = -\EE_{\bw\sim\pi_\tau}[\nabla_{\bz}\varphi(\bzx_\tau,\bw)\ba_\tau^T(\bw)] \bpx_\tau.
\end{align}
Using the definition of $H$, it is easy to verify that $\bz_\tau^x$ and $\bpx_\tau$ satisfy the dynamic equations
stated above.
\end{proof}

{
\begin{proposition}
The gradient flow of the flow-based random feature model \eqref{eqn: compositional-random-feature}  is given by 
\begin{equation}\label{eqn: gradient-flow-fbrf}
    \partial_t \ba_{\tau}(\bw,t) = - \frac{\delta \CR}{\delta \ba} = - \EE_{\bx}[\varphi(\bz_{\tau}^{\bx}(t),\bw)\bp_{\tau}^{\bx}(t)],
\end{equation}
where $\bz^{\bx}(t)$ and $\bp^{\bx}(t)$ are the state and co-state at time $t$ generated by $\ba(\cdot, \cdot)$ through Eqn. \eqref{eqn: forward-back-fbrf}.
\end{proposition}
}
Note that  for each value of $\tau$, there is a gradient flow equation for $\ba_{\tau}$. The coupling between different values of $\tau$'s
is  through the Eqn. \eqref{eqn: forward-back-fbrf}.

\begin{assumption}\label{assumption: feature-fbrf}
    Assume that $\varphi=\varphi(\bz,\bw)$ is continuous with respect to $\bz,w$, and  there is a constant $C$ such that $\max\{|\varphi(\bz,\bw)|, \|\nabla_{\bz} \varphi(\bz,\bw)\|\}\leq C$. Moreover,  assume that  the family $\{\sum_{k=1}^m a_k \varphi(\bz,\bw_k)\}$
    has the universal approximation property, namely  any continuous function can be uniformly approximated by 
    functions of the form $\{\sum_{k=1}^m a_k \varphi(\bz,\bw_k)\}$.  
\end{assumption}
The above assumption holds for $\varphi(\bz,\bw) = \sigma(\bz\cdot\bb+c)$ with $\bw=(\bb,c)\in\SS^{D}$ (the unit sphere in $\R^D$)
 and $\sigma(t)=\tanh(t)$.

\begin{proposition}
Assume that $f^*(\bx)$ is continuous with respect to $\bx$, then under Assumption \ref{assumption: feature-fbrf}  any stationary point $\ba$ of $\CR$ that satisfies $\EE_{\bw\sim\pi_\tau}[\|\ba_\tau(\bw)\|]<\infty$ is also a global minimum.
\end{proposition}
\begin{proof}
By definition, the following holds for any $\bw\in\Omega$
\[
    \frac{\delta \CR}{\delta \ba} = \EE_{\bx}[\varphi(\bz_\tau^{\bx},\bw)\bp_\tau^{\bx}]  = 0.
\]
Therefore, for any $\{\bw_k\}_{k=1}^m$ we have 
\[
\EE_{\bx}[\sum_{k=1}^m a_k\varphi(\bz_\tau^{\bx},\bw_k)\bp_\tau^{\bx}] = 0.
\]
From the universal approximation property, we obtain
\[
\EE_{\bx}[g(\bz^{\bx}_\tau)\bp_\tau^{\bx}] = 0
\]
 for any continuous function $g$.

Let $\bu(\bz,\tau) = \EE_{\bw\sim\pi_\tau}[\ba(\bw)\varphi(\bz,\bw)]$.  $\bz_1^{\bx}$  is then given by the flow map of the following ODE
\begin{equation}\label{eqn:ODE-uniqueness}
\begin{aligned}
    \bz_0 &= V{\bx}\\
    \frac{d\bz}{d\tau} &= \bu(\bz,\tau).
\end{aligned}
\end{equation}
The assumption implies that $\|\bu(\bz,\tau)\|\leq C$ and $\|\nabla_{\bz} \bu(\bz,\tau)\|\leq C$. By the Picard-Lindelof theorem, the solution of ODE
 \eqref{eqn:ODE-uniqueness} is unique. 
Since $\text{rank}(V)=d+1$,  the mapping $\bx \to \bz_\tau^{\bx}$ is non-degenerate. Therefore,  $g(\bz_{\tau}^{\bx})$ can represent any continuous function of $\bx$. Hence, the following holds for any continuous function $h$
\begin{align}
    \EE_{\bx}[h(\bx)\bp_1^{\bx}]=0,
\end{align}

Next, the stability of the forward ODE implies that $\bz_1^{\bx}$  is continuous with respect to $\bx$. Along with the assumption that $f^*(\bx)$ and $\ell'(\cdot,\cdot)$ are continuous,  we conclude that  $\bp_1^{\bx}=\bm{1}\ell'(\bm{1}^T\bz_1^{\bx},f^*(\bx))$ is continuous with respect to $\bx$. Taking $h(\bx)=\bp_1^{\bx}$ leads to 
\[
    \EE_{\bx}[\|\bp_1^{\bx}\|^2]=0.
\]
This implies that  $\bp_1^{\bx} = \bm{1} \ell'(f(\bx;\ba),f^*(\bx))=0$ almost surely, which implies that $\PP_{\bx}\{\ell'(f(\bx;\ba),f^*(\bx))=0\}=1$. Consequently,  $f(\bx;\ba)=f^*(\bx)$ almost surely.

\end{proof}

The proposition above is concerned with the stationary points of the loss functional.
We now turn to the stationary points of the gradient flow.

\begin{assumption}\label{assumption: flow-based-random-feature}
\begin{enumerate}
    \item Assume that $\Omega=\SS^D$ and $\pi_1$ is absolute continuous with respect to the Lebesgue measure on $\SS^D$. Moreover,  assume that $\pi_1$ has a continuous, positive density on $\bw\in\SS^D$. 
    \item For any $\tau_1,\tau_2\in [0,1]$, \[
    d_{TV}(\pi_{\tau_1}, \pi_{\tau_2}):=\inf_{\|f\|_{\infty}\leq 1} \EE_{\bw\sim\pi_{\tau_1}}[f(\bw)] - \EE_{\bw\sim\pi_{\tau_2}}[]f(\bw)]
    \rightarrow 0,
    \]
    as $\tau_2-\tau_1\to 0$.
\end{enumerate}
\end{assumption}

\begin{proposition}
The dissipation of the gradient flow \eqref{eqn: gradient-flow-fbrf} is given by 
\begin{align}\label{eqn: compositional-random-feature-dissipation}
    \frac{d\CR}{dt} &=  - \int_0^1 \EE_{\bw\sim \pi_\tau}\left\|\frac{\delta \CR}{\delta \ba}\right\|_2^2 d\tau \\
    &=- \int_0^1 \EE_{\bw\sim\pi_\tau}[\|\EE_{\bx}[\varphi(\bz_{\tau}^{\bx}(t),\bw))\bp_{\tau}^{\bx}(t)]\|_2^2] d\tau.
\end{align}
Moreover, under Assumption \ref{assumption: flow-based-random-feature}, let $\ba$ be a stationary point of the gradient flow, i.e. $d\CR/dt=0$ and assume that $\int_0^1 \EE_{\bw\sim\pi_\tau}\|\ba_\tau(\bw)\|_2^2 d \tau <\infty$, then we have
\[
    \frac{\delta \CR}{\delta \ba} =0.
\]
\end{proposition}

\begin{proof}
Let $g(\bw,\tau) = \|\EE_{\bx}[\varphi(\bz_\tau^x,\bw)\bpx_\tau]\|_2^2$. Then any stationary point of the gradient flow satisfies
\begin{equation}\label{eqn: rand-feature-stationary-condition}
    \frac{d\CR}{dt} = - \int_0^1 \EE_{\bw\sim\pi_\tau}g(\bw,\tau) d\tau = 0
\end{equation}
We first show that $g(\bw,\tau)$ is continuous with respect to $\tau$.
Let $(\bz^x_\tau, \bp^{x}_\tau)$ be the solution of \eqref{eqn: forward-back-fbrf}. 
By definition, we have for any $\tau_1,\tau_2\in [0,1]$, 
\be
\begin{aligned}
    \|\bz_{\tau_2}^{\bx} - \bz_{\tau_1}^{\bx}\| &\leq  \int_{\tau_1}^{\tau_2}\EE_{\bw\sim\pi_s}[\|\ba_s(\bw)\|_2|\varphi(z_s^{\bx},\bw)|]ds\\
    &\leq C \int_{\tau_1}^{\tau_2}\sqrt{\EE_{\bw\sim\pi_s}\|\ba_s(\bw)\||^2_2} ds\\
    &\leq C \left((\tau_2-\tau_1)\int_{\tau_1}^{\tau_2} \EE_{\bw\sim\pi_s}\|\ba_s(\bw)\|_2^2ds\right)^{1/2}\\
    &\leq C \left((\tau_2-\tau_1)\int_{0}^{1} \EE_{\bw\sim\pi_s}\|\ba_s(\bw)\|_2^2ds\right)^{1/2},
\end{aligned}
\ee
where the second inequality follows from the fact that $|\varphi|\leq C$. Since $\int_0^1 \EE_{\bw\sim\pi_s}[\|a_s(\bw)\|_2^2]ds <\infty$, the above upper bound implies that $\bz_\tau^{\bx}$ is continuous with respect to $\tau$. Similarly, we can prove that $\bpx_\tau$ is also continuous with respect to $\tau$. 

By definition 
\be
\begin{aligned}
    g(\bw,\tau) &\leq C\EE_{\bx}\|\bp_\tau^{\bx}\|_2^2 \leq C \|e^{\int_{\tau}^1 Q_s^Tds}\|\\ &\leq C e^{\int_0^1 \|Q_s^T\|_2 ds}\leq Ce^{\int_0^1 \EE_{\bw\sim\pi_s}\|a_s(\bw)\|ds}\leq C.
\end{aligned}
\ee
This implies that $g(\bw,\cdot)$ is uniformly bounded. Therefore, we have
\be
\begin{aligned}
    |\EE_{\bw\sim\pi_{\tau_2}}&g(\bw,\tau_2) - \EE_{\bw\sim\pi_{\tau_1}}g(\bw,\tau_1)| \\
    &\leq  |\EE_{\bw\sim\pi_{\tau_2}}g(\bw,\tau_2) - \EE_{\bw\sim\pi_{\tau_1}}g(\bw,\tau_2)| + |\EE_{\bw\sim\pi_{\tau_1}}g(\bw,\tau_2) - \EE_{\bw\sim\pi_{\tau_1}}[g(\bw,\tau_1)]|\\
    &\leq C d_{TV}(\pi_{\tau_2}, \pi_{\tau_1}) + \max_{w}|g(\bw,\tau_2)-g(\bw,\tau_1)| \\
    &= o(1)
\end{aligned}
\ee
as $\tau_2-\tau_1\to 0$.
The above inequality implies that $\EE_{\bw\sim\pi_\tau}g(\bw,\tau)$ is continuous with respect to $\tau$. Using the fact that $g(\bw,\tau)\geq 0$ and the stationarity condition \eqref{eqn: rand-feature-stationary-condition}, we  conclude that 
\[
    \EE_{\bw\sim\pi_\tau}g(\bw,\tau) = 0.
\]
Since $g(\bw,\tau)$ is continuous with respect to $\bw$ and  $\pi_1$ has full support,   we  have for all $\bw\in\SS^D$. 
\[
    \frac{\delta \CR}{\delta \ba}(\bw,\tau) = g(\bw,\tau)=0.
\]
\end{proof}

\subsection{Gradient flow for the flow-based neural networks}\label{sec: flow-neural-net}
Consider the following flow-based model 
\be\label{eqn: resnet-model}
\begin{aligned}
\bz_0^{\bx} &= V \tilde{\bx} \\
\frac{d\bz_\tau^{\bx}}{d\tau} &= \EE_{\bw\sim\pi_\tau}[\bm{\varphi}(\bz,\bw)],\quad  \forall \tau\in [0,1]\\
f(\bx;\pi) &:= \bm{1}^T\bz_1^{\bx},
\end{aligned}
\ee
where  $\bw\in\Omega$ and $\bvarphi: \RR^D\times \Omega \mapsto \RR^D$.
Here  the parameters are $\pi=(\pi_\tau)_{\tau\in [0,1]}$, a one-parameter family of probability measures. 

To derive the gradient flow for  \eqref{eqn: loss-functional}, we first need to define the parameter space appropriately.  Denote by $X := \{ \pi : [0,1]\mapsto \cP_2(\Omega)\}$, the space of all  feasible parameters. For any $\pi^1,\pi^2\in X$, consider the following metric:
\[
    d^2(\pi^1,\pi^2): = \int_0^1 W_2^2(\pi_\tau^1,\pi^2_\tau) d\tau,
\]
where $W_2(\cdot,\cdot)$ is the 2-Wasserstein distance.  { In this case, the Hamiltonian is given by 
\[
    H(\bz,\bp,\mu) := \EE_{\bw\sim\mu}[\bp^T\bvarphi(\bz,\bw)].
\]
}

\begin{proposition}\label{pro: gradient-flow-resnet}
The gradient flow in the metric space $(X,d)$ for the objective function \eqref{eqn: loss-functional} is given by
{
\begin{align}\label{eqn: flow-resnet}
\partial_{t} \pi_\tau(\bw,t) = \nabla_{\bw} \cdot\big(\pi_\tau \nabla_{\bw} V_\tau(\bw;\pi)]\big), \,\, \forall \tau \in [0,1],
\end{align}
where 
\[
V_\tau(\bw;\pi) = \EE_{\bx}[\frac{\delta H}{\delta\mu}(\bz^{\bx}_\tau(t),\bp^{\bx}_\tau(t), \pi_\tau(\cdot,t))] = \EE_{\bx}[(\bp^{\bx}_\tau(t))^T\bvarphi(\bz^{\bx}_\tau(t),\bw)]
\]
}
and for each $\bx$, $(\bz_\tau^{\bx}(t), \bp_\tau^{\bx}(t))$ satisfies 
\begin{equation}\label{eqn: forward-back}
\begin{aligned}
\frac{d\bz_\tau^{\bx}(t)}{d\tau} &= \nabla_{\bp} H = \EE_{\bw\sim\pi_\tau}[\bvarphi(\bz_\tau^{\bx}(t),\bw)] \\
\frac{d\bp_\tau^{\bx}(t)}{d\tau} &= - \nabla_{\bz} H =- \EE_{\bw\sim\pi_\tau} [\nabla_{\bz} \bvarphi(\bz_\tau^{\bx}(t),\bw) \bp_\tau^{\bx}(t)].
\end{aligned}
\end{equation}
with the boundary conditions 
$
    \bz_0^{\bx}(t) = V \tilde{\bx},
    \bp_1^{\bx}(t) = \bm{1} \ell'(f(\bx;\pi(\cdot,t)), f^*(\bx))
$.
\end{proposition}

For this gradient flow,  the energy dissipation relation is given by 
\begin{align}
\frac{d \CR}{d t} = - \int_0^1 \EE_{\bw\sim\pi_\tau(\cdot;t)}[\|\EE_{\bx}\nabla_{\bw}^T \bm{\varphi}(\bz_\tau^{\bx}(t),\bw)) \bp_\tau^{\bx}(t)\|^2]d\tau.
\end{align}

\paragraph*{Heuristic argument for  Proposition \ref{pro: gradient-flow-resnet}}.
To simplify notations, we first ignore the superscripts $\bx$. 
 Consider the following dynamics
\begin{align}
\frac{d\bz_\tau}{d\tau} &= \EE_{\bw\sim \pi_\tau}[\bvarphi(\bz_\tau,\bw)],\\
\frac{d\tilde{\bz}_\tau}{d\tau} &= \EE_{\bw\sim \pi_\tau+\varepsilon \delta_\tau}[\bvarphi(\tilde{\bz}_\tau,\bw)],
\end{align}
where the second equation is the dynamics generated by the perturbed parameter $\pi + \varepsilon \delta$.
Let $\Delta_\tau = \tilde{\bz}_\tau - \bz_\tau$,  we have 
\begin{align}
\frac{d \Delta_\tau}{d\tau} &= \EE_{\bw\sim\pi_\tau}[\bvarphi(\tilde{\bz}_\tau,\bw)-\bvarphi(\bz_\tau,\bw)] + \varepsilon\EE_{\bw\sim \delta_\tau}[\bvarphi(\tilde{\bz}_\tau,\bw)] \\
&= \EE_{\bw\sim\pi_\tau}[\nabla_{\bz}^T \bvarphi(\bz_\tau,\bw)] \Delta_\tau + \varepsilon\EE_{\bw\sim \delta_\tau}[\bvarphi(\tilde{\bz}_\tau,\bw)] + o(\varepsilon) 
\end{align}
Let $Q_\tau=\EE_{\bw\sim\pi_\tau}[\nabla_{\bz}^T \bvarphi(\bz_\tau,\bw)]$. Then,
integrating the ODE along with the  condition that $\Delta_0=0$, we get 
\begin{equation}\label{eqn: xxx}
\begin{aligned}
\Delta_1 &= \varepsilon \int_0^1 e^{\int_\tau^1 Q_s ds}\EE_{\bw\sim\delta_\tau} [\bvarphi(\tilde{\bz}_{\tau},\bw)] d\tau+ o(\varepsilon)\\
&= \varepsilon \int_0^1 e^{\int_\tau^1 Q_s ds} \EE_{\bw\sim\delta_\tau}[\bvarphi(\bz_\tau,\bw)]d\tau + o(\varepsilon).
\end{aligned}
\end{equation}
Now we have 
\begin{align}\label{eqn: varitional-resnet}
\nonumber \CR(\pi+\varepsilon \delta) - \CR(\pi) &= \ell(\bm{1}^T\tilde{\bz}_1,f^*(\bx)) - \ell(\bm{1}^T\bz_1,f^*(\bx))\\
\nonumber &= \varepsilon \langle \bm{1}\ell'(\bm{1}^T\bz_1,f^*(\bx)), \Delta_1\rangle  + o(\varepsilon) \\
\nonumber &= \varepsilon \left\langle \bm{1}\ell'(\bm{1}^T\bz_1,f^*(\bx)), \int_0^1 e^{\int_\tau^1 Q_s ds}\EE_{\bw\sim\delta_\tau} [\bvarphi(\tilde{\bz_\tau},\bw)]d\tau\right\rangle  + o(\varepsilon)\\
&= \varepsilon \int_0^1 \EE_{\bw\sim\delta_\tau}[\bp_\tau^T\bvarphi(\bz_\tau,\bw)] d\tau + o(\varepsilon),
\end{align}
where we have defined $\bp_\tau=e^{\int_\tau^1 Q_s^T ds} \bm{1} \ell'(\bm{1}^T\bz_1, f^*(\bx))$, and it  satisfies the following backward equation 
\begin{align}
\bp_1 &= \bm{1} \ell'(\bm{1}^T\bz_1^{\bx}, f^*(\bx)) \\
\frac{d\bp_\tau}{d\tau} &= - \EE_{\bw\sim\pi_\tau}[\nabla_{\bz}^T\bvarphi(\bz_\tau,\bw)]\bp_\tau = -\partial_{\bz} H(\bz_\tau, \bp_\tau,\pi_\tau).
\end{align}
Moreover, from the definition of $H$, it is easy to see that 
\[
    H(\bz,\bp,\mu+\varepsilon \delta) - H(\bz,\bp,\mu) = \varepsilon \EE_{\bw\sim \delta}[\bp^T\bvarphi(\bz,\bw)] + o(\varepsilon).
\]
Plugging the above equation into Eqn. \eqref{eqn: varitional-resnet} leads to 
\begin{equation}
\CR(\pi+\varepsilon \delta) - \CR(\pi) = \int_0^1 [H(\bz_\tau,\pi_\tau+\varepsilon \delta_\tau,\bp_\tau) - H(\bz_\tau,\pi_\tau,\bp_\tau)] d\tau + o(\varepsilon).
\end{equation}

Now we turn to the derivation of the gradient flow, defined as the limit of the generalized minimizing movements (GMM) scheme \cite{ambrosio2008gradient,santambrogio2017euclidean}:
\begin{equation}\label{eqn: rest-gmm}
\begin{aligned}
\pi^{n+1} &= \argmin \CR(\pi) + \frac{d(\pi,\pi^n)}{2\varepsilon}\\
&= \argmin \CR(\pi) - \CR(\pi^n) + \frac{d(\pi,\pi^n)}{2\varepsilon} \\
&= \argmin \int_0^1 [H(\bz_\tau^n,\bp_\tau^n,\pi_\tau) - H(\bz_\tau^n,\bp_\tau^n,\pi_\tau^n)] d\tau + \frac{\int_0^1 W_2^2(\pi_\tau,\pi_\tau^n)d\tau}{2\varepsilon} + o(\varepsilon) \\
&= \argmin \int_0^1 \left([H(\bz_\tau^n,\bp_\tau^n,\pi_\tau) - H(\bz_\tau^n,\bp_\tau^n, \pi_\tau^n)] + \frac{W_2^2(\pi_\tau,\pi_\tau^n)}{2\varepsilon} \right)  d\tau + o(\varepsilon).
\end{aligned}
\end{equation}
Therefore, we have for any $\tau\in [0,1]$ 
\[
    \pi^{n+1}_\tau = \argmin_{\mu} H(\bz_\tau^n,\bp_\tau^n,\mu) + \frac{W_2^2(\mu,\pi_\tau^n)}{2\varepsilon} + o(\varepsilon).
\]
The limit of the above  is exactly the 2-Wasserstein gradient flow for minimizing $H(\bz_\tau,\bp_\tau,\mu)$.
This gives us 
\[
    \partial_{t}\pi_\tau(\bw,t) = \nabla \cdot (\pi_\tau \nabla \frac{\delta H}{\delta \mu}(\bz_\tau(t),\bp_\tau(t), \bw)).
\]
Lastly, taking  expectation with respect to $\bx$, we complete the proof.
\qed

\begin{remark}
To make this argument rigorous, we need to establish the existence and uniqueness of the limit  
of GMM \eqref{eqn: rest-gmm}. This is a lengthy but straightforward argument. We will leave the details to interested reader. 
\end{remark}

Similar results have also been independently obtained in \cite{jabir2019mean}.

\section{Discretizations}

There are two kinds of discretization: discretization in the real space for the
probability distribution $\mu$ and discretization in the parameter space for the
variational problem and the flow.
The discretization in the real space is relatively straightforward for a typical supervised
or unsupervised learning problem. For problems in reinforcement learning or solving
PDEs, this can be more tricky. However, we will skip this issue here and leave it for
future work. Instead, we will focus on the discretization in the parameter space.

There are also two levels of discretization:  One can either discretize the variational problem for the loss function
and use one's favorite optimization algorithm on the discretized problem, or one can discretize the continuous 
integral-differential equation for the training dynamics.  
We will focus on the latter.

 \subsection{Recovering the two-layer neural network model}
Consider the functions admitting the following expectation representation,
\begin{equation}\label{eqn: transf-model}
    f(\bx;\pi) = \EE_{\bw\sim\pi} [\varphi(\bx;\bw)].
\end{equation}
The corresponding gradient flow is given by 
\be \label{eqn: trans-model-gradient-flow}
    \partial_t \pi_t  = \nabla \cdot (\pi_t \bv(\pi_t,\bw)),
\ee 
where $\bv$ is the velocity field given by 
\be\label{eqn: velocity-field}
\bv(\pi,\bw)= \EE_{\bx}[(f(\bx;\pi)-f^*(\bx))\nabla_{\bw}\varphi(\bx;\bw)].
\ee

Let us consider the simplest particle method discretization of the model \eqref{eqn: transf-model}
and the gradient flow \eqref{eqn: trans-model-gradient-flow}. We approximate $\pi$ by 
\be\label{eqn: particle-approx}
    \hat{\pi} (\bw) = \frac{1}{m}\sum_{k=1}^m \delta(\bw-\bw_k).
\ee
Here $m$ is the number of particles,  and $\{\bw_k\}_{k=1}^m$ are the $m$ particles. In this approximation, the evolution of $\hat{\pi}$ will be completely determined by the $m$ particles.

First, the function represented by $\hat{\pi}$ is given by 
\be 
    \begin{aligned}
    f(\bx;\hat{\pi}) &= \EE_{w\sim\hat{\pi}}[\varphi(\bx;\bw)] \\
                &=\frac{1}{m}\sum_{k=1}^m \varphi(\bx;\bw_k)
    \end{aligned}
\ee 
The weak formulation of \eqref{eqn: trans-model-gradient-flow} is given by 
\be 
    \frac{d}{dt}\int g(\bw) d\pi_t(\bw) = \int \nabla \cdot (\pi_t \bv(\pi_t,\bw)) g(\bw)  d\bw = - \int \langle \bv(\pi_t,\bw), \nabla g(\bw)\rangle d\pi_t(\bw),
\ee 
where $g$ is a test function.  Plugging \eqref{eqn: particle-approx} into the above equation, we get 
\be 
\frac{d}{dt} \frac{1}{m}\sum_{k=1}^m g(\bw_k)=\frac{1}{m}\sum_{k=1}^m \langle \nabla g(\bw_k),\frac{d\bw_k}{dt}\rangle =  - \frac{1}{m}\sum_{k=1}^m \langle \nabla g(\bw_k), \bv(\hat{\pi}_t,\bw_k) \rangle.
\ee 
Therefore, the dynamics of the particles follows 
\be 
    \frac{d\bw_k}{dt} = - \bv(\hat{\pi}_t,\bw_k).
\ee 

If we taking $\varphi(\bx;\bw) = a\sigma(\bb^T\bx)$ with $\bw:=(a,\bb)$, the particle method discretization is given by 
\begin{equation}\label{eqn: two-layer-net}
\begin{aligned}
    f(\bx;\hat{\pi}) &= \frac{1}{m}\sum_{k=1}^m a_k \sigma(\bb_k^T\bx) \\
    \frac{d a_k}{dt} &= - \EE_{\bx}[(f(\bx;\hat{\pi}_t)-f^*(\bx))\sigma(\bb_k^T\bx)] \\
    \frac{d \bb_k}{dt} &= - \EE_{\bx}[(f(\bx;\hat{\pi}_t)-f^*(\bx))a_k\sigma'(\bb_k^T\bx)\bx].
\end{aligned}
\end{equation}
This is the the (continuous time) gradient descent dynamics for the {\it scaled} (i.e. with the factor $1/m$ in front of the expression for $f$)
two layer neural network model.

It can be shown that the dynamics described above is exactly the same as that of the GD for scaled two-layer neural networks.
In fact, we have:

\begin{lemma}
Given a set of initial data $\{\bw_k^0,  k \in [m] \}$.
The solution of \eqref{eqn: trans-model-gradient-flow} with initial data
$\pi(0) = \frac 1m \sum_{k=1}^m \delta_{\bw_k^0}$ is given by
$$\pi(t) = \frac 1m \sum_{k=1}^m \delta_{\bw_k(t)}
$$
where $\{\bw_k(\cdot), k \in [m] \}$
solves the following systems of ODEs:
\begin{equation*}
\label{GD-NN}
\frac{d \bw_k}{dt} = - \bv(\pi_t,\bw_k), \quad
\bw_k(0) = \bw_k^0, \quad k \in [m]
\end{equation*}
\end{lemma}

In particular, this lemma says that the continuous flow equation \eqref{eqn: trans-model-gradient-flow}
also holds for the discrete case with a finite set of neurons.

\subsection{A smoothed particle method}

A popular modification of the particle method is the smoothed particle method.
Here we illustrate how one can formulate the smoothed particle method
for the integral transform-based model \eqref{eqn: transf-model} and 
the gradient flow \eqref{eqn: trans-model-gradient-flow}.
We will consider the special case when $\phi(\bx;\bw)=a\sigma(\bb^T\bx)$. Here $\bw=(a,\bb)$ and $\sigma(t)=\max(0,t)$ is the ReLU activation function.

Consider a smoothed particle approximation \cite{monaghan2005smoothed} to $\pi_t$ \footnote{This also coincides with the 
Gaussian mixture approximation suggested by Jianfeng Lu.}
\[
\hat{\pi}_t(\bw)=\frac{1}{m}\sum_{k=1}^m \phi_h(\bw-\bw_k(t)),
\]
where $\phi_h$ is the probability density function of $\cN(0, h^2I)$. 
The smoothed particle discretization of the flow-based model and the gradient flow is given by 
\begin{align}\label{eqn: sm-f}
\nonumber  f(\bx;\hat{\pi}_t) &= \EE_{\bw\sim\hat{\pi}_t}[\phi(\bx;\bw)]\\
                        &= \frac{1}{m}\sum_{k=1}^m \EE_{\bm{\xi}}[\phi(\bx;\bw_k+h\bm{\xi})]\\
    \frac{d\bw_k}{dt} &= \nonumber \EE_{\bw\sim\phi_h(\cdot-\bw_k(t))}[\bv(\hat{\pi}_t,\bw)] \\ \label{eqn: sm-g}
                    &= \EE_{\bm{\xi}}[\bv(\hat{\pi}_t,\bw_k+h\bm{\xi})]
\end{align}
where $\bm{\xi}\sim\cN(0,I_{d+1})$. The right hand side of the last equality is the smoothed velocity. 

For this to be a practical numerical algorithm, we need a way to evaluate the terms in \eqref{eqn: sm-f} and \eqref{eqn: sm-g}.
We will defer this to a future publication.
To get some insight about the nature of this smooth particle method, we consider the special case when
the data lies on the sphere, i.e. $\|\bx\|=1$.

Write $\bm{\xi}=(\xi_1,\bm{\xi}_2)$ with $\xi_1\in\RR$ and $\bm{\xi}_2\in\RR^d$, then the smoothed particle method becomes 
\begin{align}
\nonumber    f(\bx;\hat{\pi}) &= \EE_{(a,\bb)\sim\hat{\pi}} [a\sigma(\bb^T\bx)]\\
\nonumber    &= \frac{1}{m}\sum_{k=1}^m \EE_{(\xi_1,\bm{\xi}_2)}[(a_k+ h\xi_1)\sigma((\bb_k+ h \bm{\xi}_2)^T\bx)]\\
\nonumber    &= \frac{1}{m}\sum_{k=1}^m a_k\EE_{\bm{\xi}_2}[\sigma(\bb_k^T\bx+ h \bm{\xi}_2^T\bx)]\\
    &= \frac{1}{m}\sum_{k=1}^m a_k\EE_{\xi\sim\cN(0,1)}[\sigma(\bb_k^T\bx+ h \xi)],
\end{align}
where in the last equation we have used the assumption that $\|\bx\|=1$. Define a new activation function 
\begin{align}\label{eqn: smoothed-relu}
 \nonumber   \sigma_h(t) &= \EE_{\xi\sim\cN(0,1)}[\sigma(t+h\xi)] = \int_{-t/h}^{\infty} (t+h\xi) \frac{1}{\sqrt{2\pi}}e^{-\xi^2/2}d\xi.\\
    &= t\Phi(\frac{t}{h}) + h \phi(\frac{t}{h}),
\end{align}
where $\phi, \Phi$ are the probability density  and cumulative density functions of the standard normal distribution, respectively. 
Then the discretized model can be rewritten as 
\be\label{eqn: smooth-two-layer}
    f(\bx;\hat{\pi}) = \frac{1}{m}\sum_{k=1}^m a_k \sigma_h(\bb_k^T\bx).
\ee 
This is  a new two-layer neural network with  activation function $\sigma_h$, which can be viewed  as a ``smoothed'' ReLU.
It is easy to see    $\sup_{t\in\RR} |\sigma_h(t) -\sigma(t)| = O(h)$.  Figure \ref{fig: activation} shows the difference between the two activation functions.

From Eqn. \eqref{eqn: velocity-field} and \eqref{eqn: sm-g}, we see that the evolution of particles follows
\begin{align}
\nonumber    \frac{d a_k}{dt} &= \EE_{\bm{\xi}}[\EE_{\bx}[(f(\bx;\hat{\pi})-f^*(\bx))\sigma(\bb_k^T\bx+h\bm{\xi_2}^T\bx)]] \\
    &= \EE_{\bx}[(f(\bx;\hat{\pi})-f^*(\bx))\sigma_h(\bb_k^T\bx)]\\
\nonumber \frac{d\bb_k}{dt} &= \EE_{\bm{(\xi_1,\bm{\xi}_2)}}[\EE_{\bx}[(f(\bx;\hat{\pi})-f^*(\bx))(a+h\xi_1)\sigma'(\bb_k^T\bx+h\bm{\xi}_2^T\bx)\bx]] \\
&= \EE_{\bx}[(f(\bx;\hat{\pi})-f^*(\bx))a \sigma_h'(\bb_k^T\bx)\bx].
\end{align}
This is exactly the  gradient descent dynamics for  the two-layer smoothed ReLU network  \eqref{eqn: smooth-two-layer}.
\begin{figure}[!h]
    \centering
    \includegraphics[width=0.4\textwidth]{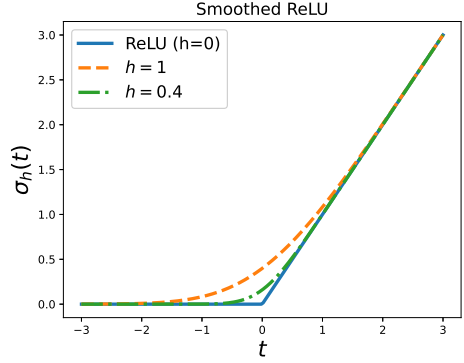}
    \vspace*{-4mm}
    \caption{\small Comparison between ReLU and the smoothed ReLU activation functions.}
    \label{fig: activation}
\end{figure}

    It should be noted that the dominate term $t\Phi(t/h)$ in \eqref{eqn: smoothed-relu} is exactly the activation function Gaussian Error Linear Unit (GELU) \cite{hendrycks2016gaussian}, which  has become quite popular recently \cite{devlin2018bert}.

\subsection{A new algorithm for integral transform-based models}
Consider the representation
\begin{align}
f(\bx; a, \pi) = \int a(\bb) \sigma(\bb^T\bx) d\pi(\bb).
\end{align}
We now view both $a$ and $\pi$ as parameters.

For the loss functional \eqref{eqn: loss-functional}, we have  
\be 
\begin{aligned}
\frac{\delta \CR}{\delta a} &= \EE_{\bx}[(f(\bx;a,\pi)-f^*(\bx))\sigma(\bb^T\bx)] \\
\frac{\delta \CR}{\delta \pi}&=\EE_{\bx}[(f(\bx;a,\pi)-f^*(\bx))a(\bb)\sigma(\bb^T\bx)].
\end{aligned}
\ee
It is tricky to design a particle method  for the combined model A and model B dynamics for this problem
\eqref{eqn: A-dynamics} and  \eqref{eqn: B-dynamics}.
 Therefore we consider instead the modified  "gradient flow":
\begin{align}\label{eqn: nonconservatie-gflow-two-layer}
\partial_t a_t +\bv(a_t,\pi_t,\bb) \cdot \nabla a &=  - \frac{\delta \CR}{\delta a} ,\\
\partial_t \pi_t &=- \nabla \cdot \big(\pi_t \bv(a_t,\pi_t,\bb) \big),
\end{align}
where 
\begin{align}
 \nonumber   \bv(a,\pi,\bb) &= - \nabla \frac{\delta R}{\delta \pi}(a,\pi,\bb) \\
    &= -\EE_{\bx}[(f(\bx;a,\pi)-f^*(\bx))(\nabla a(\bb) \sigma(\bb^T\bx) + a(\bb)\sigma'(\bb^T\bx)\bx].
\end{align}

Let $\hat{\pi}_t = \frac{1}{m}\sum_{j=1}^m \delta(\cdot-\bb_j(t))$. For each particle $\bw_j$,  define two quantities: 
\[
a_j(t) = a(\bb_j(t),t), \quad \bu_j = \nabla a(\bb_j(t),t).
\]
Denote by $\hat{\ba}=\{a(\bb_j(t), t)\}_{j=1}^m$.
Then the  function represented by $\hat{a}$ and $\hat{\pi}$ is given by 
\begin{equation}
f(\bx;\hat{a},\hat{\pi}) = \frac{1}{m}\sum_{j=1}^m a_j (t) \sigma(\bb_j(t)^T\bx).
\end{equation}
It is now straightforward to derive the dynamics for the particle method:
\begin{align}
 \nonumber   \frac{d \bb_j}{dt} &= \bv(\hat{a}_t,\hat{\pi}_t,\bb_j) \\
    &= - \EE_{\bx}[(f(\bx;\hat{a}_t,\hat{\pi}_t)-f^*(\bx)))(\bu_j \sigma(\bb^T_j\bx)+a_j\bx \sigma'(\bb_j^T\bx))], \label{eqn: u_particle_w}
\end{align}  
\begin{align}
\nonumber    \frac{d a_j}{dt} &= \frac{d a(\bb_j(t),t)}{dt} = \langle \nabla a(\bb_j,t), \dot{\bb}_j \rangle  + \partial_t a(\bb_j, t)\\
    &= -\EE_{\bx}[(f(\bx;\hat{a}_t,\hat{\pi}_t)-f^*(\bx))\sigma(\bb^T_j\bx)] \label{eqn: u_particle_a}\\
  \nonumber  \frac{d \bu_j}{dt} &= \nabla\frac{d a(\bb_j(t),t)}{dt} \\
    &= - \EE_{\bx}[(f(\bx;\hat{a}_t,\hat{\pi}_t)-f^*(\bx))\sigma'(\bb_j^T\bx)\bx] \label{eqn: u_particle_u}
\end{align}

\begin{figure}[!h]
    \centering
    \includegraphics[width=0.39\textwidth]{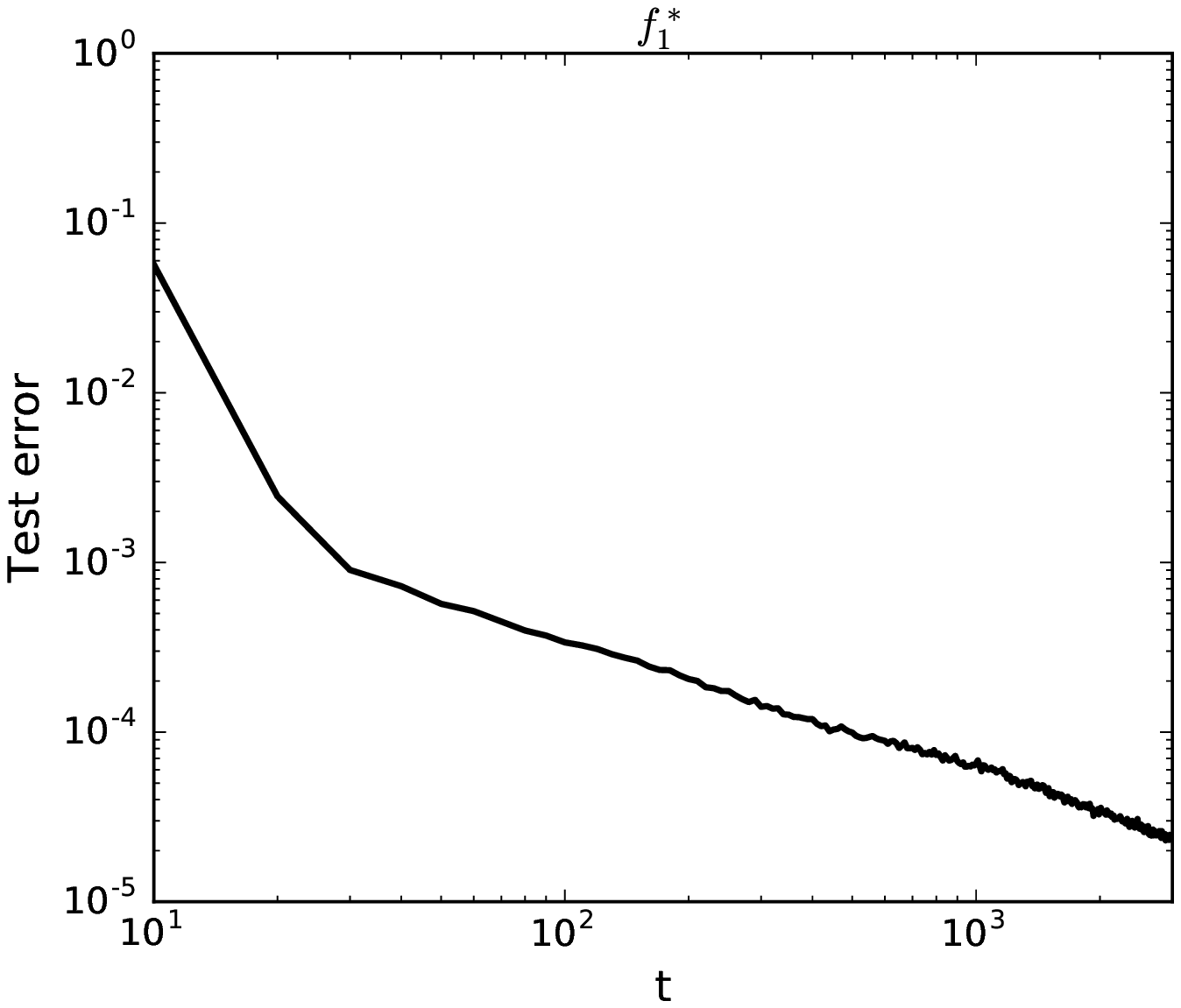}
    \includegraphics[width=0.39\textwidth]{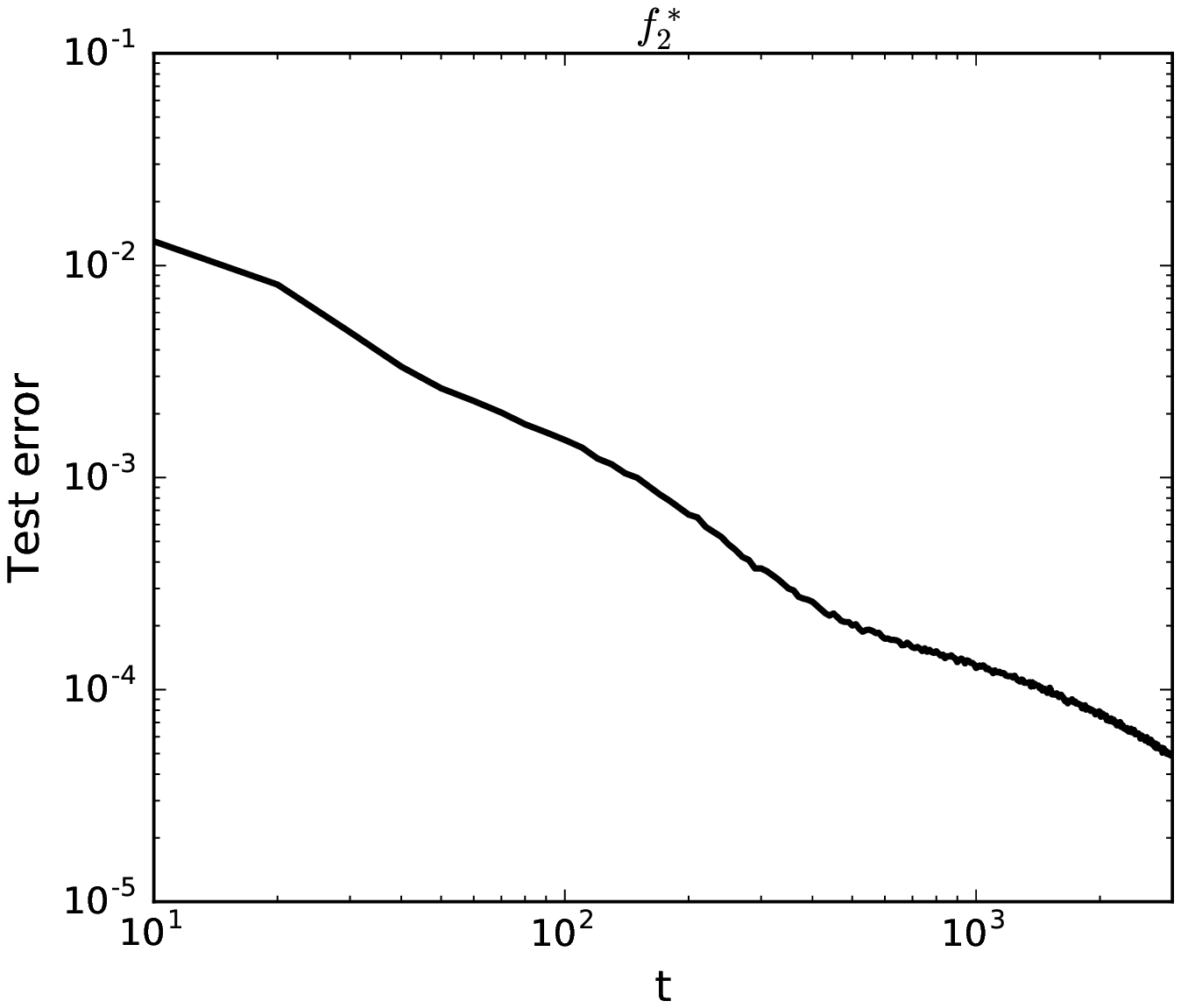}
    \vspace*{-3mm}
    \caption{\small The testing error along the dynamics of the particle discretization in~\eqref{eqn: u_particle_a}, \eqref{eqn: u_particle_u} and~\eqref{eqn: u_particle_w}.  The figure on the left is the result for a target function in the RKHS.
    The figure on the right is the result for a target function outside of the RKHS.  
    See the main text for the details of the target functions.}
    \label{fig: u_particle}
\end{figure}

To see that this is a reasonable algorithm, we report the results of some preliminary numerical experiments for this 
particle method.
We let $\bx\in \RR^{10}$ and $\sigma$ be the ReLU activation function. 
 $a_j$ and $\bu_j$ are initialized from $0$ and $\bw_j$ is initialized from a standard Gaussian.
 We  take $m=1000$. 
 For simplicity, we approximate the expectations in~\eqref{eqn: u_particle_a}, \eqref{eqn: u_particle_u} and~\eqref{eqn: u_particle_w} using $100$ online samples at each step. Hence there is no generalization gap. The forward Euler scheme is used to  numerically solve the ODEs and the step size is $0.1$. 
 We consider two kinds of target functions. The first one is given by
\begin{equation}
    f_{1}^*(\bx) = \sum_{i=1}^{10} c_i K(\bx,\bx_i),
\end{equation}
where $
    K(\bx,\bx') = \int_{\SS^{9}} \sigma(\bw^T\bx)\sigma(\bw^T\bx')\pi(d\bw),
$
with $\pi$ being the uniform distribution on $\SS^9$,
$c_i\in\RR$, $\bx_i\in\RR^{10}$ are randomly sampled. 
This function belongs to the reproducing kernel Hilbert space (RKHS) associated with $K$.
The second one is a function outside of the RKHS given by  
\[
f^*_2(\bx)=\sum_{i=1}^{10} \sigma(\bw_i^T\bx),
\] 
where $\{\bw_i\}$ are randomly drawn from $\pi$.

The results for this experiment are reported in Figure~\ref{fig: u_particle}. Figure~\ref{fig: u_particle} shows that for both target functions,
the testing errors decrease nicely in the rate $O(1/t)$ during the training process.

\section{The generalization error}

Let $\cH,\cH_m$ denote the spaces of functions represented by the continuous and discretized model, respectively. Here $m$ denotes the number of grid points or particles in the discretization. 
Let $S$ denote the training set and $|S|=n$. Denote by $\hat{f}_{m,n,t}$ the solution generated by the gradient descent 
dynamics at time $t$,  and let
\begin{align}
    \CR(\hat{f}_{m,n,t}) = \EE_{\bx}[\ell(\hat{f}_{m,n,t}(\bx),f^*(\bx))].
\end{align}

One way to address the generalization problem is to look for an  estimate of the following type
\be\label{eqn: apriori-estimate}
\CR(\hat{f}_{m,n,t})\leq e(1/m,1/n,t,\|f^*\|),
\ee
where $\|f^*\|$ is some norm of the target function.  

There are two ways to obtain  estimates of the type in \eqref{eqn: apriori-estimate}. One is through the a priori estimates of the discretized gradient descent dynamics. The other is through the a priori estimates of the gradient flow, i.e. the PDEs.  
 In the following, we illustrate both approaches using the random feature model.
We recover results proved in  \cite{carratino2018learning} with simpler arguments.

\paragraph{Learning with random features. } 
Assume that the target function is given by
\[
f^*(\bx)=\EE_{\bb\sim\pi}[a^*(\bb)\varphi(\bx;\bb)],
\]
where $\pi$ is a fixed probability distribution. Assume that $|\varphi(\bx;\bb)|\leq 1$. The RKHS norm of $f^*$ is given by
\[
    \|f^*\|_{\cH} := \sqrt{\EE_{\bb\sim\pi}[|a^*(\bb)|^2]}.
\]
We also make the following assumption 
\[
    \|f^*\|_{\infty} = \text{ess sup}_{\bb} |a^*(\bb)| \in [1, \infty).
\]
The particle method discretization is given by 
\begin{align}
    f_m(\bx;\ba,\Bb^0) = \frac{1}{m}\sum_{j=1}^m a_j \varphi(\bx;\bb_j^0),
\end{align}
where $\ba=(a_1,\dots,a_m)^T\in\RR^m$ is the parameter to be learned and $\Bb^0=(\bb_1^0,\dots,\bb_m^0)^T$ are randomly
sampled from $\pi$. 
The loss function is defined to be
\be 
    \hat{\CR}_n(\ba) = \frac{1}{n}\sum_{i=1}^n \ell(f(\bx_i;\ba,\Bb^0),f^*(\bx_i)).
\ee 
The gradient descent is then given by 
\[
    \frac{d\ba}{dt} = - \nabla \hat{\CR}_n(\ba).
\]
We assume that $\ba_0=\bm{0}$.

The subtlety of the problem can be appreciated from the  work of \cite{belkin2019reconciling} 
which shows that the generalization error of this model can be very large in the regime where $m\approx n$.

\subsection{Analyzing the discretized model}
We  decompose the generalization error into two terms:
\begin{align} \label{eqn: gen-err-decomp-1}
    \CR(\hat{f}_{m,n,t}) =   \underbrace{\hat{\CR}_n(\hat{f}_{m,n,t})}_{I_1}  +  \underbrace{\CR(\hat{f}_{m,n,t})- \hat{\CR}_n(\hat{f}_{m,n,t})}_{I_2}.
\end{align}
Here $I_1,I_2$ are the optimization (training) error and generalization gap, respectively. 

The general philosophy is that the generalization gap is bounded by a term of the form $\|f_{m, n, t} \|/\sqrt{n}$.
Here $\| \cdot \|$ is some norm determined by the model.
For example, for random feature models, this is the RKHS norm.
For two-layer neural network models, this is the Barron norm \cite{e2018priori}.
Therefore to estimate the generalization gap, one needs to derive a priori bounds on these norms.


For any $\bar{\ba}\in\RR^m$, define 
\be 
    J(t) := t(\hat{\CR}_n(\ba_t)-\hat{\CR}_n(\bar{\ba})) + \frac{1}{2}\|\ba_t - \bar{\ba}\|_2^2.
\ee 
Since $\hat{\CR}_n(\ba)$ is convex, we have $dJ/dt\leq 0$. So $J(t)\leq J(0)$, i.e. 
\[
    t(\hat{\CR}_n(\ba_t)-\hat{\CR}_n(\bar{\ba})) + \frac{1}{2}\|\ba_t - \bar{\ba}\|_2^2\leq \frac{1}{2}\|\ba_0 - \bar{\ba}\|_2^2.
\]
This gives
\begin{equation}\label{eqn: GDF-random-feature}
\begin{aligned}
    \hat{\CR}_n(\ba_t) &\leq \frac{\|\bar{\ba}\|^2_2}{2t} \\
    \|\ba_t\|_2^2 &\leq 2\|\bar{\ba}\|_2^2 + 2t\hat{\CR}_n(\bar{\ba}).
\end{aligned}
\end{equation}
The first inequality gives a bound on the training error.
The second inequality provides a bound for the norm of the parameters.

Using \eqref{eqn: gen-err-decomp-1} and the Rademacher complexity \cite{bartlett2002rademacher} bound for the generalization gap (see Eqn. (93-95) in \cite{ma2019comparative}),   we have the following i estimates. 

\begin{proposition}\label{thm: aposteriori}
For any $\delta\in (0,1)$, with probability $1-\delta$ over the training examples, we have 
\begin{align}\label{eqn: aposteriori}
 \CR(\ba)\lesssim \hat{\CR}_n(\ba) + \frac{\|\ba\|^2/m+\|f^*\|_\cH^2}{\sqrt{n}}\left(1+\sqrt{\log((\|\ba\|/\sqrt{m}+1)^2/\delta)}\right).
\end{align}
\end{proposition}
\begin{proof}
Let $\cF_C :=\{f_m(\cdot;\ba,\Bb^0): \|\ba\|/\sqrt{m}\leq C\}$ and  $\cH_C:=\{(f_m(\cdot;\ba,\Bb_0)-f^*)^2: \|\ba\|/\sqrt{m}\leq C\}$. 
By Cauchy-Schwarz inequality,  $|f_m(\bx;\ba,B_0)|\leq C$ and $|f^*(\bx)|\leq \sqrt{\int a(\bb)^2d\pi(\bb)}\leq \|f^*\|_{\cH}$. Hence,  $g(t)=(t-y_i)^2$ is $2(C+\|f^*\|_\cH)$-Lipschitz continuous. Then by the contraction property of Rademacher complexity, we have 
\begin{align*}
    \rad_n(\cH_C)\leq 2(C+\|f^*\|_\cH)\rad_n(\cF_C)\leq \frac{2C(C+\|f^*\|_\cH)}{\sqrt{n}},
\end{align*}
where the last inequality follows from the fact that 
$
    \rad(\cF_C)\leq \frac{C}{\sqrt{n}}
$ \cite{shalev2014understanding}.
Hence,  with probability $1-\delta$ we have for any $\ba$ satisfying  $\|\ba\|/\sqrt{m}\leq C$,
\begin{align}
    \CR(\ba)&\lesssim \hat{\CR}_n(\ba)+\frac{(C+\|f^*\|_\cH)C}{\sqrt{n}} + (\|f^*\|_\cH+C)^2\sqrt{\frac{\log(2/\delta)}{n}}\\
    &\lesssim \hat{\CR}_n(\ba) + \frac{C^2+\|f^*\|_\cH^2}{\sqrt{n}}\left(1+\sqrt{\log(1/\delta)}\right).
\end{align}

Using the union bound, we obtain 
\begin{align*}
 \CR(\ba)&\lesssim \hat{\CR}_n(\ba) + \frac{\|\ba\|^2/m+\|f^*\|_\cH^2}{\sqrt{n}}\left(1+\sqrt{\log((\|\ba\|/\sqrt{m}+1)^2/\delta)}\right)\\
 &\leq \hat{\CR}_n(\ba) + \frac{\|\ba\|^2/m+\|f^*\|_\cH^2}{\sqrt{n}}\left(1+\sqrt{\log(2(\|\ba\|^2/m+1)/\delta)}\right)
\end{align*}

\end{proof}

The following proposition provides a bound on the approximation error on finite training samples. 

\begin{proposition}\label{pro: approx-empirical-error}
Given the training set $\{(\bx_i,f(\bx_i))\}_{i=1}^n$, for any $\delta\in (0,1)$, we have that with probability at least $1-\delta$ over the sampling of $\Bb^0$, there exists $\tilde{\ba}\in\RR^m$ such that 
\begin{align}
\hat{\CR}_n(\tilde{\ba})&\leq \frac{2\log(2n/\delta)}{m} \|f\|_{\cH}^2  +  \frac{8\log^2(2n/\delta)}{9m^2} \|f\|^2_{\infty},\\
    \frac{\|\tilde{\ba}\|^2}{m} &\leq \|f\|_\cH^2 + \sqrt{\frac{\log(2/\delta)}{2m}}\|f\|^2_{\infty}
\end{align}
\end{proposition}
\begin{proof}
Define the exception set:
\[
S_i(q) :=\Big\{\{\bb_j^0\}_{j=1}^m: \big|\frac{1}{m}\sum_{j=1}^m a(\bb_j^0)\varphi(\bx_i;\bb_j^0) - f^*(\bx_i)\big|\geq q\Big\}.
\]

Let $Z_{j}^i=a(\bb_j^0)\varphi(\bx_i;\bb_j^0)-f^*(\bx_i)$. Then  $\EE[Z_j^i]=0, \text{Var}[Z_j^i]\leq \|f\|_{\cH}^2$ and $|Z_j^i|\leq 2\|f\|_{\infty}$. By Bernstein inequality, we have
\[
\PP\{S_i(q)\}=\PP\Big\{|\frac{1}{m}\sum_{j=1}^m Z_j^i|\geq q\Big\}\leq 2\exp\left(-\frac{mq^2}{\|f\|_\cH^2+2\|f\|_\infty q/3}\right).
\]
Hence, 
\begin{align}
    \PP\{\hat{\CR}_n(\ba)\geq q^2\}\leq \sum_{i=1}^n \PP\{S_i(q)\}\leq 2n\exp\left(-\frac{mq^2}{\|f\|_\cH^2+2\|f\|_\infty q/3}\right).
\end{align}
In addition, by Hoeffding's  inequality and  $|a(\bb_j^0)|^2\leq \|f\|_\infty^2$, we have 
\begin{align}
    \PP\Big\{\big|\frac{1}{m}\sum_{j=1}^m a(\bb_j^0)^2 - \|f\|_\cH^2\big|\geq q\Big\}\leq 2\exp\left(-\frac{2mq^2}{\|f\|_{\infty}^4}\right).
\end{align}
Hence, taking $\tilde{\ba}=a(\Bb^0):=(a(\bb_1^0),\dots,a(\bb_m^0))^T$, we have with probability  $1-\delta$ over the random sampling of $\Bb^0$,
\begin{align}
    \hat{\CR}_n(a(\Bb^0))&\lesssim \frac{2\log(2n/\delta)}{m} \|f\|^2_\cH  +  \frac{8\log^2(2n/\delta)}{9m^2} \|f\|^2_{\infty},\\
    \frac{\|a(\Bb^0)\|^2}{m} &\leq \|f\|_\cH^2 + \sqrt{\frac{\log(2/\delta)}{2m}}\|f\|^2_{\infty}
\end{align}
\end{proof}
Combing Proposition \ref{thm: aposteriori} and Proposition \ref{pro: approx-empirical-error}, we have the following a priori estimates of the generalization error of  GD solutions. 

\begin{theorem}
For any $\delta\in (0,1)$, assume that $m\geq \log^2(n/\delta)$. With probability $1-\delta$, we have 
\[
\CR(\ba_{m^2t})\lesssim \left(\|f^*\|_\cH^2+\sqrt{\frac{\log(2/\delta)}{m}}\|f^*\|_\infty^2\right) \varepsilon(n,m,t,\delta,\|f^*\|_\infty),
\]
where 
\[
\varepsilon(n,m,t,\delta) = \frac{1}{mt}+\frac{1}{\sqrt{n}}\left(1+t\log(n/\delta)\right)\left(1+\sqrt{\log(n/\delta)}+\sqrt{\log(t\|f^*\|_\infty/\delta)} \right).
\]
\end{theorem}

\begin{proof}
Taking $\bar{\ba}$ be the solution constructed in Proposition \ref{pro: approx-empirical-error} and plugging into Eqn. \eqref{eqn: GDF-random-feature},  we then have
\[
\hat{\CR}_n(\ba_{m^2t}) \leq \frac{\|\bar{\ba}\|^2}{2m^2t}\lesssim \frac{\|f^*\|_\cH^2+\sqrt{\frac{\log(2/\delta)}{m}}\|f^*\|_\infty^2}{mt},
\]
and 
\begin{align*}
\frac{\|\ba_{m^2t}\|^2}{m}&\leq \frac{2\|\bar{\ba}\|^2}{m} + \frac{m^2t\hat{\CR}_n(\bar{\ba})}{m}\\
&\lesssim \|f\|_\cH^2+\sqrt{\frac{\log(2/\delta)}{m}} \|f\|_{\infty}^2 + mt\left(\frac{\log(2n/\delta)}{m} \|f\|_\cH^2  +  \frac{\log^2(2n/\delta)}{m^2} \|f\|_{\infty}^2\right)\\
&= \big(1+\log(n/\delta) t\big)\|f\|_\cH^2 + \left(\sqrt{\frac{\log(2/\delta)}{m}}+\frac{t\log^2(2n/\delta)}{m}\right)\|f\|_{\infty}^2\\
&= Q_{1} \|f\|_{\cH}^2 + Q_2 \|f\|_{\infty}^2,
\end{align*}
where $Q_1=1+\log(n/\delta) t, Q_2 = \sqrt{\log(2/\delta)/m}+ t\log^2(2n/\delta)/m$.
Plugging the above estimates into Eqn. \eqref{eqn: aposteriori}, we obtain 
\begin{align*}
\CR(\ba_{m^2t})&\lesssim 
    \hat{\CR}_n(\ba_{m^2t}) + \frac{\|\ba_{m^2t}\|^2/m+\|f^*\|^2_\cH}{\sqrt{n}}\left(1 + \sqrt{\log(2\|\ba_{m^2t}\|^2/m+1)/\delta)}\right)\\
    &\lesssim \frac{\|f^*\|_\cH^2+\sqrt{\frac{\log(2/\delta)}{m}}\|f^*\|_\infty^2}{mt} \\
    &\quad + \frac{1}{\sqrt{n}}((Q_1+1)\|f\|_\cH^2+Q_2\|f\|_\infty^2)\left(1+\sqrt{\log((Q_1+Q_2)\|f\|_\infty^2+1)/\delta)}\right)\\
    &\lesssim \|f^*\|_\cH^2 I_1 + \|f^*\|_\infty^2\sqrt{\frac{\log(2/\delta)}{m}} I_2,
\end{align*}
where in the second inequality we used the fact that $\|f\|_{\cH}\leq \|f\|_{\infty}$. 
Using the definition of $Q_1, Q_2$ and $m\gtrsim \log^2(n/\delta)$ gives us that 
\begin{align*}
    I_1&\leq \frac{1}{mt}+\frac{1}{\sqrt{n}}\left(1+t\log(n/\delta)\right)\left(1+\sqrt{\log(n/\delta)}+\sqrt{\log(t\|f^*\|_\infty/\delta)} \right)\\
    I_2 &\leq \frac{1}{mt} + \frac{1}{\sqrt{n}}\left(1+\frac{t\log^{3/2}(n/\delta)}{\sqrt{m}}\right)\left(1+\sqrt{\log(n/\delta)}+\sqrt{\log(t\|f^*\|_\infty/\delta)} \right).
\end{align*}
Moreover,  it follows from  $m\gtrsim \log^2(n/\delta)$ that $I_2\lesssim I_1$. 
This completes the proof.
\end{proof}

\subsection{Analyzing the continuous model}

The approach presented above is the standard approach in machine learning theory.
It works since the loss functional is convex in this case.  It is difficult to generalize this to more complicated situations
due to the lack of convexity.
Here we explore an alternative approach by studying  the continuous problem.
Our hope is that some of the PDE techniques can be leveraged to help our understanding. 
{ One such example is found in \cite{chizat2018global}, which proves a global convergence  result for the gradient flow for two-layer neural networks 
by analyzing the PDE.}

We decompose the generalization error as follows, 
\begin{align}\label{error}
  \CR(\hat{f}_{m,n,t}) &=  \CR(\hat{f}_{m,n,t}) -  \CR(\hat{f}_{\infty,n,t})  \\
\label{eqn: conti-gen-bound-2}     & +  \CR(\hat{f}_{\infty,n,t}) -  \hat{\CR}_n(\hat{f}_{\infty,n,t}) \\
\label{eqn: conti-gen-bound-3}     & +  \hat{\CR}_n(\hat{f}_{\infty,n,t}),
\end{align}
{where $f_{\infty, n, t}$ is the solution given by the gradient flow of the continuous model.}
The three terms are respectively the discretization error, the generalization gap and the training error
(for the continuous problem). The latter two terms require a priori estimates of the gradient flow. 

Consider the random feature model, the gradient flow is given by 
\[
    \partial_t a(\bw,t) = - \frac{\delta \hat{\CR}_n}{\delta a}.
\]
Similar to above, we define 
\[
J(t) := t(\hat{\CR}_n(a_t)-\hat{\CR}_n(a^*)) + \frac{1}{2}\|a_t-a^*\|_{L^2(\pi)}^2.
\]
Then  we have
\begin{align}
    \frac{dJ(t)}{dt} &= - t \|\frac{\delta \hat{\CR}_n}{\delta a}\|_{L^2(\pi)}^2 +  \hat{\CR}_n(a_t)-\hat{\CR}_n(a^*) + \langle a_t -a^*, -\frac{\delta \hat{\CR}_n}{\delta a}\rangle_{L^2(\pi)}\\
    &\leq - t \|\frac{\delta \hat{\CR}_n}{\delta a}\|_{L^2(\pi)}^2\leq 0,
\end{align}
where the second inequality follows from the convexity of $\CR_n$ with respect to $a$.
It follows that 
\[
     t(\hat{\CR}_n(a_t)-\hat{\CR}_n(a^*)) + \frac{1}{2}\|a_t-a^*\|_{L^2(\pi)}^2 \leq \frac{1}{2}\|a_0-a^*\|_{L^2(\pi)}^2.
\]
Since $\ba_0=0$ and $\hat{\CR}_n(a^*)=0$, we get
\begin{align}
    \hat{\CR}_n(a_t) &\leq \frac{\|a^*\|_{L^2(\pi)}^2}{2t} \\
    \|a_t\|_{L^2(\pi)}&\leq 2\|a^*\|_{L^2(\pi)}.
\end{align}
Let $c=2\|a^*\|_{L^2(\pi)}$.
Then the function $\hat{f}_{\infty, n,t}$ must lie in  $\cF_c = \{f_{\infty}(\cdot;a): \|a\|_{L^2(\pi)}\leq c\}$, with  $f_{\infty}(x;a):=\int a(\bb)\varphi(\bx;\bb)d\pi(\bb)$.
Let $\cH_c = \{(f_{\infty}(\cdot;a)-f^*)^2: \|a\|_{L^2(\pi)}\leq c\}$. By the contraction property of Rademacher complexity, we have 
\[
\rad_n(\cH_c)\leq 4c \rad_n(\cF_c)\leq \frac{8c^2}{\sqrt{n}}.
\]
Moreover, $|h|\leq 4 c^2$ for any $h\in \cH_c$.  

Following Eqn. \eqref{eqn: conti-gen-bound-2} \eqref{eqn: conti-gen-bound-3} and using the Rademacher complexity-based bound,  we have 
\begin{align}
    \CR(\hat{f}_{\infty,n,t}) = \CR(a_t) & \leq \hat{\CR}_n(a_t) + 2\rad_n(\cH_c) + 4c^2 \sqrt{\frac{2\log(2/\delta)}{n}}\\
   & \lesssim \frac{\|a^*\|_{L^2(\pi)}^2}{2t} + \frac{ \|a^*\|_{L^2(\pi)}^2}{\sqrt{n}} (1+\sqrt{\log(2/\delta)})\\
    &= \frac{\|f^*\|_{\cH}^2}{2t} + \frac{(1+\sqrt{\log(2/\delta)}) \|f^*\|_{\cH}^2}{\sqrt{n}}
\end{align}

The treatment of the discretization error in \eqref{error} is more complex.
This requires substantial machinery in numerical analysis. We will postpone this to future publications.

\section{An example}
In this section, we study a simple $1$-dimensional case of the integral
transform-based model proposed in Section~\ref{ssec:transform}. 
Specifically, we consider the following conservative gradient flow,
\begin{equation}\label{eqn: 1d_meanfield}
    \frac{\partial \rho_t}{\partial t} = \nabla\cdot \left(\rho_t \nabla \int K(w, w')(\rho_t(dw')-\rho^*(dw'))\right),
\end{equation}
where $w\in[0, 2\pi]$, $\rho^*$ is a fixed probability distribution that
determines the target function:
$$
f^*(x) = \int_0^{2 \pi} \sigma(\cos(w - x)) \rho^*(dw)
$$
$\rho_t$ obeys the periodic boundary condition. 
$K$ is given by
$$
K(w, w') = \frac 1{2\pi} \int_0^{2 \pi} \sigma(\cos(w - x)) \sigma(\cos(w'-x)) dx
$$
It is easy to see that $K$ can be written as
$K = K(w-w')$ . Hence~\eqref{eqn: 1d_meanfield} can be written 
in a convolutional form,
\begin{equation}\label{eqn: 1d_meanfield2}
    \frac{\partial \rho_t}{\partial t} = \nabla\cdot \left(\rho_t \nabla K*(\rho_t-\rho^*)\right).
\end{equation}
{In the following analysis we consider the case where} $K$ is  positive definite, i.e.,
\begin{equation}
    \int (K*\nu) \nu(dw)>0
\end{equation}
holds for any { measure} $\nu$.
This condition is easily satisfied in practice.
In addition, we assume that $K$ is {three-times differentiable and its derivatives are bounded}.

\subsection{Global convergence for uniform target distribution}

First, we study the situation when $\rho^*$ is uniform.
In this case, one can prove global convergence
of the gradient flow~\eqref{eqn: 1d_meanfield}. 
(or free energy of~\eqref{eqn: 1d_meanfield}).

\begin{theorem}\label{thm: stationary}
Assume $\rho^*$ is the uniform distribution. 
{Let $\rho_t$ be the solution of~\eqref{eqn: 1d_meanfield2} initialized from $\rho_0$. Assume that $\rho_0$ has differentiable density function, then we have $\lim_{t\rightarrow\infty}W_2(\rho_t,\rho^*) = 0$.}
\end{theorem}

\begin{proof}
By an abuse of notation, we let $\rho_t$ and $\rho^*$ be the density function of $\rho_t$ and $\rho^*$, respectively. {First, we assume $\rho_t$ exists and has differentiable density function. 
For any probability distribution $\rho$, consider the relative entropy of $\rho$ and $\rho^*$,
\begin{equation}
    \cH(\rho|\rho^*):=\int \rho\log\left(\frac{\rho}{\rho^*}\right)dw
\end{equation}
Taking the time derivative, we have
\begin{align}
\frac{d}{dt}\cH(\rho_t|\rho^*) &=\int \partial_t\rho_t (\log\rho_t+1-\log\rho^*)dw \nonumber\\
  &= \int \nabla\cdot \left(\rho_t \nabla K*(\rho_t-\rho^*)\right)(\log\rho_t+1-\log\rho^*)dw \nonumber \\
  &= -\int \nabla K*(\rho_t-\rho^*)\nabla\rho_t dw \nonumber \\
  &= \int \Delta K*(\rho_t-\rho^*)\rho_t dw. \label{eqn: rel_ent}
\end{align}
Let $\hat{K}(k)$, $\hat{\rho}(k)$, $\hat{\rho}^*(k)$ be the coefficients of the Fourier series of $K$, $\rho$ and $\rho^*$, respectively. The Fourier expansions exist due to the differentiability assumptions on $K$ and $\rho_t$. By~\eqref{eqn: rel_ent} we have
\begin{align}
\frac{d}{dt}\cH(\rho_t|\rho^*) &= -\sum\limits_{k\in\ZZ} k^2\hat{K}(k)\hat{\rho}_t(k)(\hat{\rho}_t(k)-\hat{\rho}^*(k)) \nonumber \\
  &= -\sum\limits_{k\in\ZZ}k^2\hat{K}(k)(\hat{\rho}_t(k)-\hat{\rho}^*(k))(\hat{\rho}_t(k)-\hat{\rho}^*(k)) \label{eqn: rel_ent2}
\end{align}
The second equality of~\eqref{eqn: rel_ent2} holds because $\rho^*$ is the uniform distribution, which gives $\hat{\rho}^*(k)=0$ for all $k\neq0$. Since $K$ is positive definite, we have $\hat{K}(k)>0$ for all $k\in\ZZ$. Hence, for any $\rho\neq\rho^*$, we have
\begin{equation}
    \frac{d}{dt}\cH(\rho_t|\rho^*) < 0.
\end{equation}
Therefore, $\cH(\rho|\rho^*)$ is a Lyapunov function for the dynamics~\eqref{eqn: 1d_meanfield2}. Since the set of probability distributions $\rho$ on $[0,2\pi]$ is compact in the space $W_2$, any sublevel set of $\cH$ is compact in $W_2$. Hence, the trajectory $\rho_t$ converges to the set where $\frac{d}{dt}\cH(\rho_t|\rho^*)=0$. By~\eqref{eqn: rel_ent2}, this set contains only $\rho^*$.
This proves the statements in the theorem.

We next prove the existence and boundedness of $\partial_w\rho_t$. 
The existence and uniqueness of the solution of \eqref{eqn: 1d_meanfield2} can be proved in 
the same way as in~\cite{chizat2018global}. Therefore we only provide the main ideas here.

Taking the partial derivative with respect to $w$ 
on both sides of~\eqref{eqn: 1d_meanfield2}, and noting that $w\in\bR$, we get
\begin{equation}
\partial_t\partial_w\rho_t = \partial_w\left(\partial_w\rho_t\partial_w K*(\rho_t-\rho^*)\right)+\partial_w\left(\rho_t\partial^2_{ww}K*(\rho_t-\rho^*)\right).
\end{equation}
Hence, $\partial_w\rho_t$ is the solution of the following linear hyperbolic PDE for $u(w,t)$:
\begin{equation}
\partial_t u = \partial_w(u\partial_wK*(\rho_t-\rho^*))+u\partial^2_{ww}K*(\rho_t-\rho^*)+\rho_t\partial^3_{www}K*(\rho_t-\rho^*).
\end{equation}
By the conditions on $K$, the coefficients of the PDE above are uniformly bounded. 
Now it follows from standard PDE argument that $u$ is bounded for any finite interval of time
$[0, T]$.
}
\end{proof}


\subsection{Local convergence for the general case}

The previous global convergence result only holds for the case when $\rho^*$ is uniform. 
The next result shows that as long as $\rho_0$ is initialized close to $\rho^*$, 
the gradient flow converges to the global minimum with an $\mathcal{O}(1/t)$ rate.

\begin{theorem}
Assume the conditions of Theorem~\ref{thm: stationary} hold. Furthermore
 assume that there are constants $C_0$, $C_1$ and $C^*$ such that 
\begin{equation}\label{eqn: cond_K}
    \frac{C_0}{|k|}\leq\hat{K}(k)\leq\frac{C_1}{|k|}, {\textrm{\ and } |\hat{\rho}^*(k)|\leq\frac{C^*}{k^2}}
\end{equation}
hold for any $k\neq0$. 
Let $C$ and $t_0$ be two constants that satisfy
\begin{equation}\label{eqn: cond_local}
     { \frac{C_0t_0}{2\pi C}-32C_1(C^*t_0+C)>1,}
\end{equation}
and 
assume that $\rho_0$ satisfies 
\begin{equation}
    {|\hat{\rho}_0(k)-\hat{\rho}^*(k)|<\frac{C}{|k|^2t_0}}
\end{equation}
for any $k\neq0$.
Then we have
\begin{equation}\label{eqn: rho_bound}
    {|\hat{\rho}_t(k)-\hat{\rho}^*(k)|\leq\frac{C}{|k|^2(t+t_0)}}
\end{equation}
for any $t\geq 0$.
\end{theorem}

\begin{proof}
Let $u_t=\rho_t-\rho^*$.  By the conditions we imposed, $\hat{u}_t(0)=0$, and 
\begin{equation}
    |\hat{u}_0(k)| \leq \frac{C}{2|k|^2t_0},
\end{equation}
for any $k\neq 0$.
From equation~\eqref{eqn: 1d_meanfield2}, the dynamics of $u_t$ is 
\begin{align}\label{eqn: dyn_u}
\frac{\partial u_t}{\partial t} &=\nabla\cdot((u_t+\rho^*)\nabla K*u_t),
\end{align}
Writing~\eqref{eqn: dyn_u} in the Fourier space, we get
{ \begin{align}
\frac{d}{dt}\hat{u}_t(k) &= -\sum\limits_{l=-\infty}^\infty kl(\hat{u}_t(k-l)+\hat{\rho}^*(k-l))\hat{K}(l)\hat{u}_t(l) \nonumber \\
  &= -k^2\hat{K}(k)\hat{u}_t(k)-\sum\limits_{l\neq k}kl(\hat{u}_t(k-l)+\hat{\rho}^*(k-l))\hat{K}(l)\hat{u}_t(l), \label{eqn: u_fourier}
\end{align}
for any $k\in\ZZ$. 
Next, we consider the set 
\begin{equation}
   \cX := \left\{ (x(k))_{k=-\infty}^\infty: x(0)=0, \textrm{and } |x(k)|\leq\frac{C}{|k|^2(t+t_0)} \textrm{for }k\neq0 \right\},
\end{equation}
and show that this is an invariant set for the dynamics, i.e.
trajectories $(\hat{u}_t(k))$ initialized inside of $\cX$ will not escape from $\cX$. To 
prove this, assume that $(\hat{u}_t(\cdot))$ is at  the boundary of $\cX$, which means there exists a non-empty set $\mathcal{K}$ such that for any $k\in\mathcal{K}$ we have
\begin{equation}
    |\hat{u}_t(k)|=\frac{C}{|k^2|(t+t_0)}.
\end{equation} 
Then, for any $k\in\mathcal{K}$, 
by~\eqref{eqn: u_fourier} we have}
{\begin{align}
\frac{d}{dt}\hat{u}_{t}(k) &\leq -\frac{C_0C|k|^2}{2\pi|k||k|^2(t+t_0)}+\sum\limits_{l\neq0,k} \frac{C_1C|kl|}{|l||l^2|(t+t_0)}\left(\frac{C^*}{|k-l|^2}+\frac{C}{|k-l|^2(t+t_0)}\right) \nonumber \\
  &= -\frac{C_0C}{2\pi|k|(t+t_0)}+\frac{C_1C|k|}{(t+t_0)}\left(C^*+\frac{C}{t+t_0}\right)\sum\limits_{l\neq0, k}\frac{1}{|l^2||k_0-l|^2}. \label{eqn: contra2}
\end{align}}
For the second term on the right hand side of~\eqref{eqn: contra2}, we have
\begin{align}
\sum\limits_{l\neq0, k}\frac{1}{|l|^2|k-l|^2} &\leq 2\sum\limits_{l\geq[k/2], l\neq k}\frac{1}{|l|^2|k-l|^2} \nonumber \\
  &\leq \frac{16}{k^2}\sum\limits_{l=1}^\infty \frac{1}{l^2} \nonumber \\
  &\leq \frac{32}{k^2}. 
\end{align}
Going back to~\eqref{eqn: contra2} we have
{\begin{align}
\frac{d}{dt}\hat{u}_{t}(k) &\leq -\frac{C}{|k|^2(t+t_0)^2}\left( \frac{C_0|k|(t+t_0)}{2\pi C}-32C_1C^*(t+t_0)|k|-32C_1C|k| \right) \nonumber \\
  &\leq -\frac{C}{|k|(t+t_0)^2}\left(\frac{C_0t_0}{2\pi C}-32C_1C^*t_0-32C_1C\right), \label{eqn: contra3}
\end{align}
where the second inequality holds as a consequence of the condition~\eqref{eqn: cond_local}, which implies
\begin{equation}
    \frac{C_0}{2\pi C}-32C_1C^* > 0.
\end{equation}}
Use again condition~\eqref{eqn: cond_local} together with~\eqref{eqn: contra3}, we have
\begin{equation}\label{eqn: point_inside}
    \frac{d}{dt}\hat{u}_{t}(k) < -\frac{C}{|k|(t+t_0)^2} = \frac{d}{dt}\frac{C}{|k^2|(t+t_0)}.
\end{equation}
{Since~\eqref{eqn: point_inside} holds for any $k\in\mathcal{K}$, the vector field at $(\hat{u}_t(k))$ points inside $\cX$. Therefore, the trajectory $\{(\hat{u}_t(k)): t\geq0\}$ stays in $\cX$ for any $t>0$, which completes the proof.}
\end{proof}

\begin{remark}
The theorem above shows local convergence of the gradient descent dynamics 
with $\mathcal{O}(1/t)$ rate. For simplicity of the proof we assumed that the Fourier coefficients 
of $K$ decays with an $\mathcal{O}(1/|k|)$ rate. This condition is inessential and can be relaxed, 
at the expense of a faster decay rate imposed on $\hat{\rho}^*(k)$ and $\hat{\rho}_0(k)$. 
\end{remark}

\subsection{Numerical results}
A pseudo-spectrum method is implemented  to
numerically solve the equation~\eqref{eqn: 1d_meanfield2}.  
Specifically, we consider a $1$-D model 
\begin{equation}\label{eqn: model}
    f(x) = \int_{0}^{2\pi}\varphi(x, w)\rho(dw),
\end{equation}
with the feature $\varphi(\bx, w)$ given by
\begin{equation}\label{eqn: feature}
    \varphi(x, w) = \sum\limits_{k=-\infty}^{\infty}e^{-\frac{(x-w-2k\pi)^2}{h^2}},
\end{equation}
where $h$ is the standard deviation, and $x, w\in[0,2\pi]$. 
It is easy to see that the summation in~\eqref{eqn: feature} is finite for any $x$ and $w$, and $\varphi(x, w)$ is $2\pi$-periodic for both $x$ and $w$. 
A direct calculation gives:
\begin{equation}\label{eqn: kernel}
K(w,w') = \frac{1}{2\pi}\int_0^{2\pi}\varphi(x,w)\varphi(x,w')d x = \frac{h}{\sqrt{8\pi}}\sum\limits_{k=-\infty}^\infty e^{-\frac{(w-w'+2k\pi)^2}{2 h^2}}.
\end{equation}
In the experiments, we take $h=1$, and $\rho^*$ to be
\begin{equation}
    {\rho^*(w) = \frac{1}{2\pi}(1+0.2\times \sin(w)+0.8\times \sin(3w)).}
\end{equation}
The target function $f^*$ is displayed
in the left panel of Figure~\ref{fig: 1d_result}. 
We see that  this simple function contains three components: 
a mean value, a low frequency part (generated by $sin(w)$ in $\rho^*$), 
and a high frequency part (generated by $\sin(3s)$ in $\rho^*$). 
We take $\rho_0$ to be the uniform distribution on $[0,2\pi]$, and solve \eqref{eqn: 1d_meanfield2} 
for $10^4$ time units. The error between $f_{\rho_t}$ and $f^*$ along the path
is shown in the right panel of Figure~\ref{fig: 1d_result}. 
We see that the dynamics proceeds in three different regimes:
a nearly flat regime initially, followed by two faster regimes.
This is related to the so-called frequency principle discussed next.

\begin{figure}
    \centering
    \includegraphics[width=0.40\textwidth]{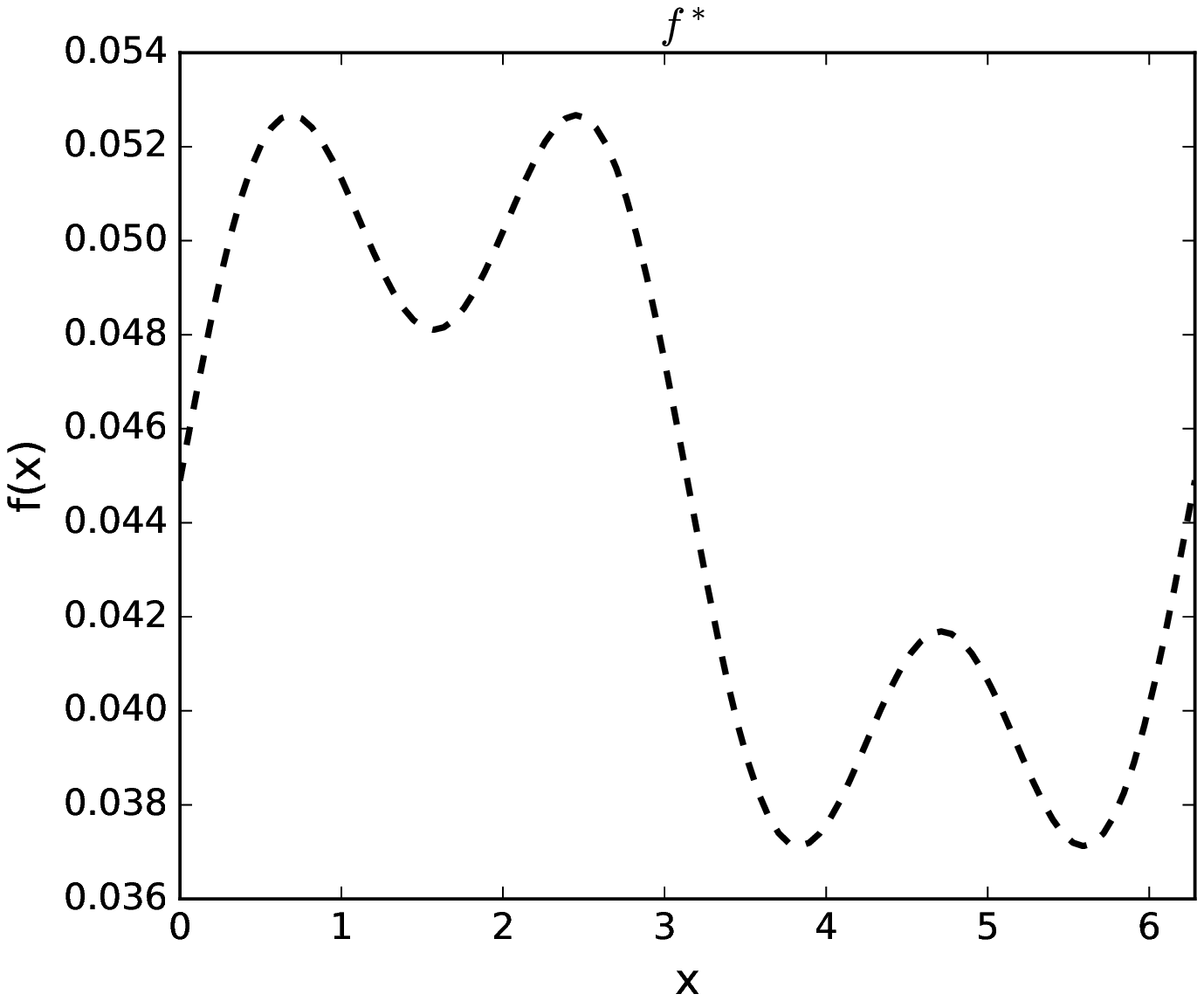}
    \includegraphics[width=0.40\textwidth]{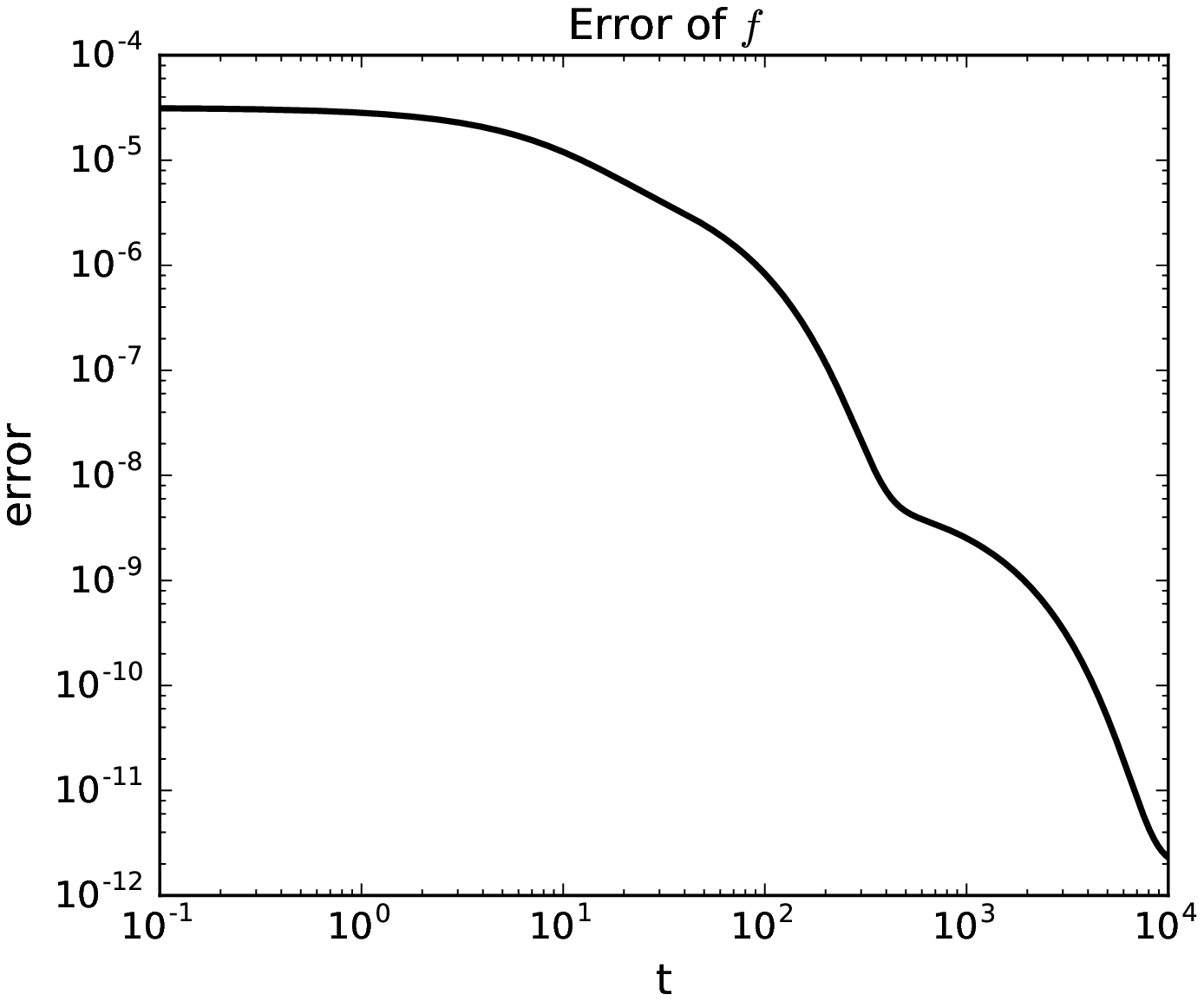}
    \vspace*{-3mm}
    \caption{\small Left: the target function $f^*$; Right: the error $\|f-f^*\|^2$ along the path. Some numerical details: $101$ Fourier components are used in the pseudo-spectral method, 
with a time step size $dt=0.01$. $4$-th order Runge-Kutta method is used for time integration.}
    \label{fig: 1d_result}
\end{figure}

 It is interesting to study the analog of the empirical risk, 
defined using the kernel:
\begin{equation}
    K_n(w,w') = \frac{1}{n}\sum\limits_{i=1}^n \varphi(x_i,w)\varphi(x_i,w')
\end{equation}
where $\{x_i\}$ is a set of data samples.
We take $n=100$ and sample the $x_i$'s from the uniform distribution on $[0,2\pi]$. The results, 
presented in Figure~\ref{fig: 1d_empirical}, suggest that the empirical loss converges to $0$, and 
the $L^2$ norm of the density function $\rho_t(w)$ stays bounded.
As was argued in the previous section, under this circumstance, the generalization error
is bounded by $C/\sqrt{n}$.

\begin{figure}
    \centering
    \includegraphics[width=0.40\textwidth]{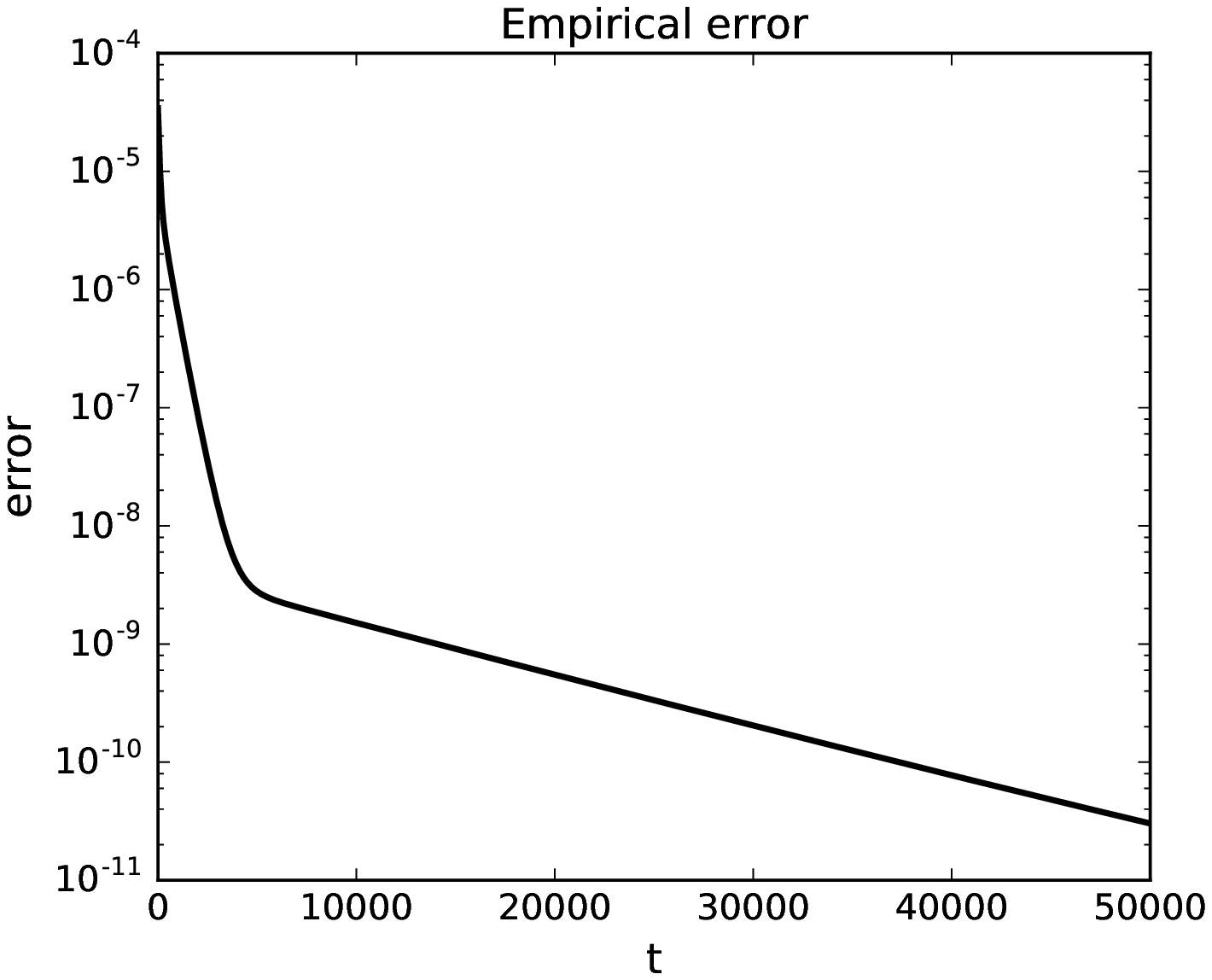}
    \includegraphics[width=0.40\textwidth]{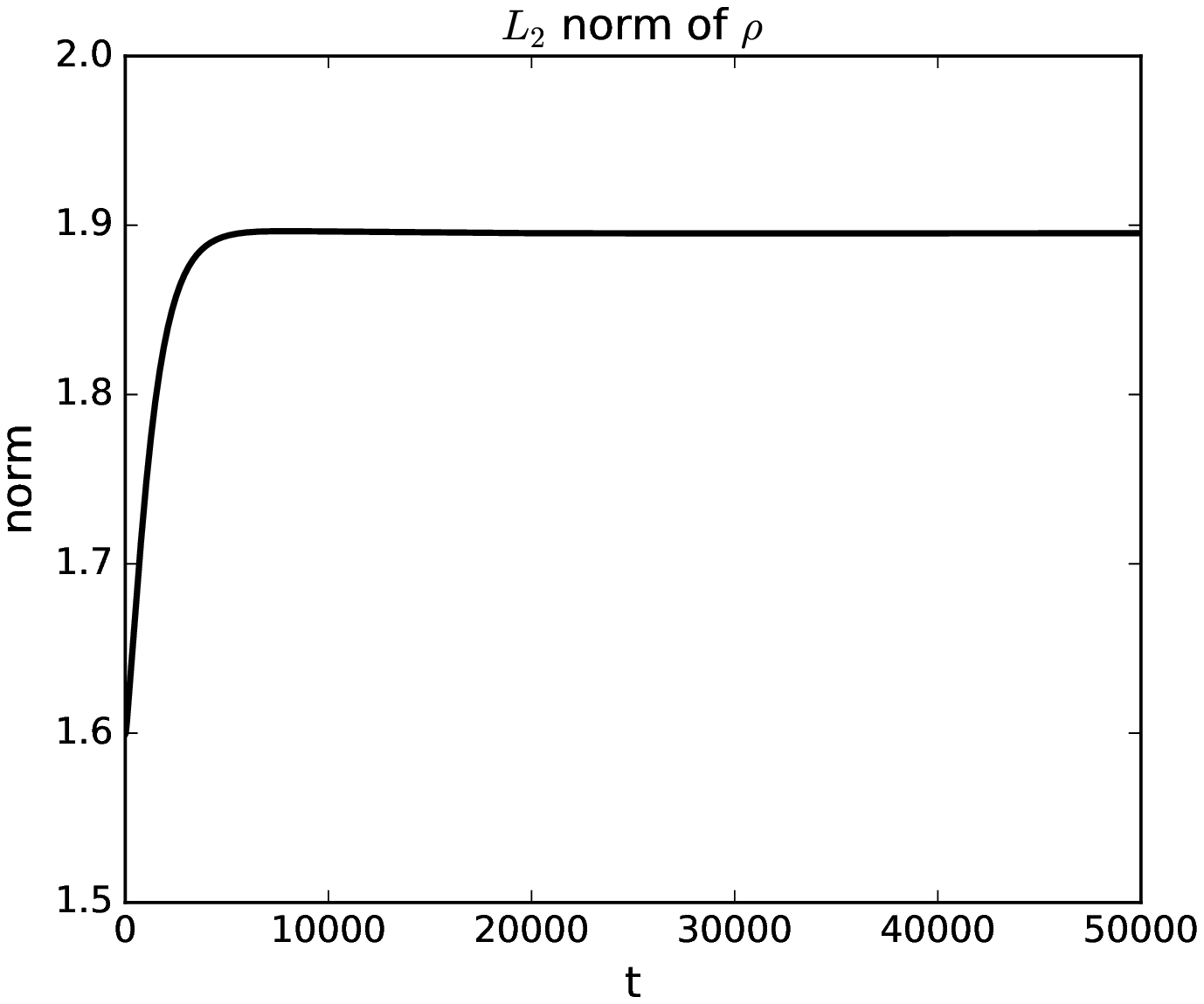}
    \vspace*{-3mm}
    \caption{\small The empirical risk (left) and the $l_2$ norm (right) of the solution 
along the gradient descent path. Here the empirical risk is used as the
free energy to define the dynamics.}
    \label{fig: 1d_empirical}
\end{figure}

\subsection{The frequency principle}
The frequency principle was suggested by Xu et al in \cite{xu2019frequency}.
The idea was that if one uses the gradient descent 
to train neural network models, then the low frequency part of the target
function is recovered before the high frequency component.
Here we examine this issue in some detail.

For this purpose it is useful to consider the dynamics in real space, i.e. 
we study the evolution of the function $f$ in~\eqref{eqn: model}.
Let $f_t$ be the function generated by $\rho_t$, we have
\begin{align}
\frac{d}{dt}f_t(x) &= \int \varphi(x,w) \partial_t\rho_t(w)dw \nonumber \\
  &= \int \varphi(x,w)\nabla\cdot\left( \rho_t(w)\int\nabla K(w,w')(\rho_t(w')-\rho^*(w')) dw'\right)dw \nonumber \\
  &= \int \varphi(x,w)\nabla\cdot\left(\rho_t(w)\int\nabla\varphi(x',w)(f_t( x')-f^*(x'))d x'\right)dw \nonumber \\
  &= -\int\left(\int\nabla\varphi(x,w)\nabla\varphi(x',w)\rho_t(w)dw\right)(f_t(x')-f^*(x'))d x'. \label{eqn: integral_eqn}
\end{align}
Therefore, the dynamics of $f_t$ is governed by an integral equation. 
This fact has important implications.
To see this more clearly, let us linearize the kernel in the above
equation around $\rho_0$, then we get
\begin{equation}
    \frac{d}{dt} f_t(x) = -\int \tilde{K}(x,x')(f_t(x')-f^*(x'))d x',
\end{equation}
where 
\begin{equation}
    \tilde{K}(x,x')=\int\nabla_w\varphi(x,w)\nabla_w\varphi(x',w)\rho_0(w)dw.
\end{equation}
By the symmetry of $x$ and $w$ in $\varphi(x,w)$, we have $\nabla_w\varphi(x,w)=-\nabla_{x}\varphi(x,w)$, and hence for $\tilde{K}$ we have
\begin{align}
\tilde{K}(x,x') &= \partial_{x}\partial_{x'}\int \varphi(x,w)\varphi(x',w)\rho_0(w)dw \nonumber \\
  &= \partial_{x}\partial_{x'} K(x,x') \nonumber \\
  &= \frac{h}{\sqrt{8\pi}}\sum\limits_{k=-\infty}^\infty \left[\frac{1}{h^2}-\frac{(x-x'+2k\pi)^2}{h^4}\right]e^{-\frac{(x-x'+2k\pi)^2}{2h^2}}, \label{eqn: tilde_K}
\end{align}
the second equality again follows from the symmetry between $x$ and $w$ in $\varphi$. 
Note that $\tilde{K}$ only depends on $x-x'$.
Hence we have the following Fourier decomposition of $\tilde{K}$:
\begin{align}
\tilde{K}(x,x') &= c_0 + \sum\limits_{k=1}^\infty b_k \sin(k(x-x'))+c_k \cos(k(x-x')) \nonumber \\
  &= c_0 +\sum\limits_{k=1}^\infty b_k(\sin(k x)\cos(k x')-\cos(k x)\sin(k x')) \nonumber \\
  &\quad + \sum\limits_{k=1}^\infty c_k(\cos(k x)\cos(k x')+\sin(k x)\sin(k x')).
\end{align}
Therefore, for any $u,v$ and $k$ we have
\begin{equation}
\int \tilde{K}(x,x')(u \sin(k x')+v \cos(k x'))d x' = 
\pi\left((c_ku+b_kv)\sin(k x)+(c_kv-b_ku)\cos(k x)\right).
\end{equation}
Consequently the eigenfunctions of $\tilde{K}$ are given by
\begin{equation}
    \left\{ u\sin(k x)+v\cos(k x): (u,v)^T \textrm{is the eigenvector of} \left[\begin{array}{cc}
        c_k  & b_k \\
        -b_k & c_k
    \end{array}\right] \right\}.
\end{equation}
The eigenvalues are $\pi\lambda_k$, where $\lambda_k = c_k + i b_k, c_k - ib_k$. 
Using \eqref{eqn: tilde_K}, we can explicitly compute the Fourier 
coefficients of $\tilde{K}$ and obtain $c_0=0$, $b_k=0$, and 
$c_k=\frac{h^2k^2}{2}e^{-h^2k^2/2}$. Hence, the eigenvalues of the operator $\tilde{K}$ are $\{\frac{\pi h^2k^2}{2}e^{-h^2k^2/2}\}$, and the eigenfunctions are 
simply the Fourier basis functions. 
{We see that the eigenvalues decrease with the frequency $k$ when $k\geq 2/h$. 
This implies that the frequency principle should hold for $k\geq 2/h$.}

Figure~\ref{fig: 1d_solution} displays the function $f_t$ at different times along the gradient 
flow path, compared to the target function $f^*$. One can see that the low frequency 
components converge faster than the high frequency components. 
The is consistent with the frequency principle. 

\begin{figure}
    \centering
    \includegraphics[width=0.35\textwidth]{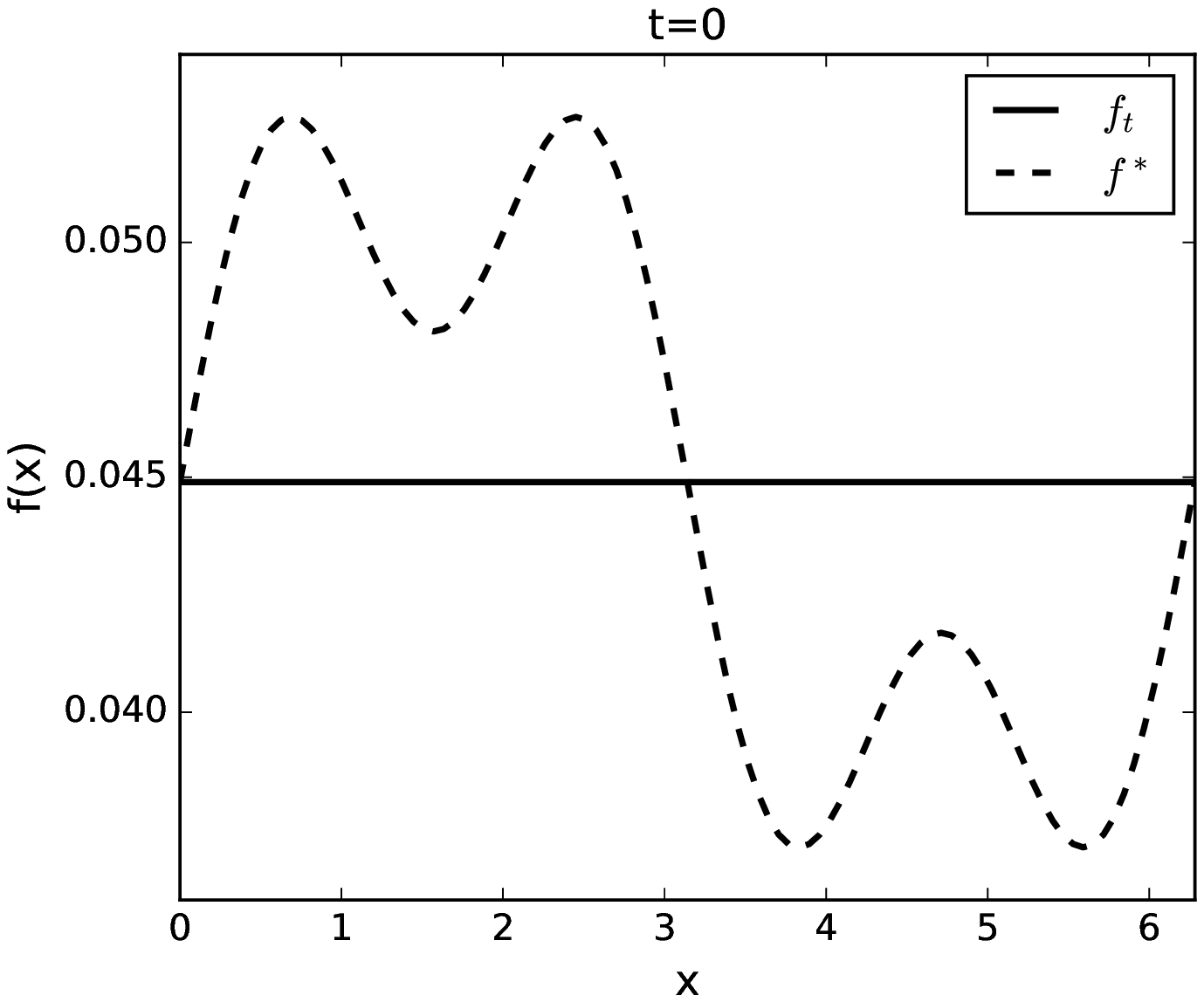}
    \includegraphics[width=0.35\textwidth]{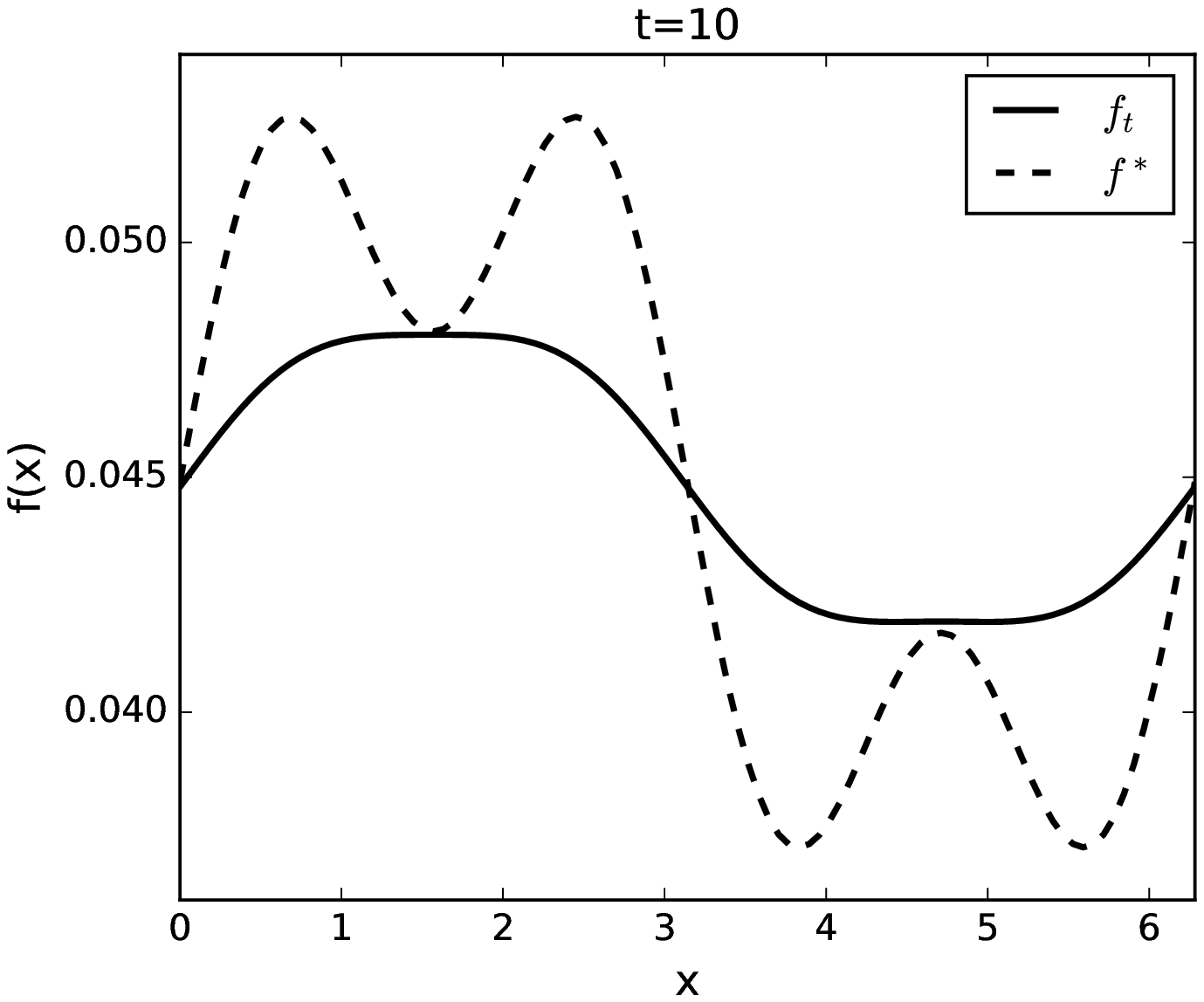}

    \includegraphics[width=0.35\textwidth]{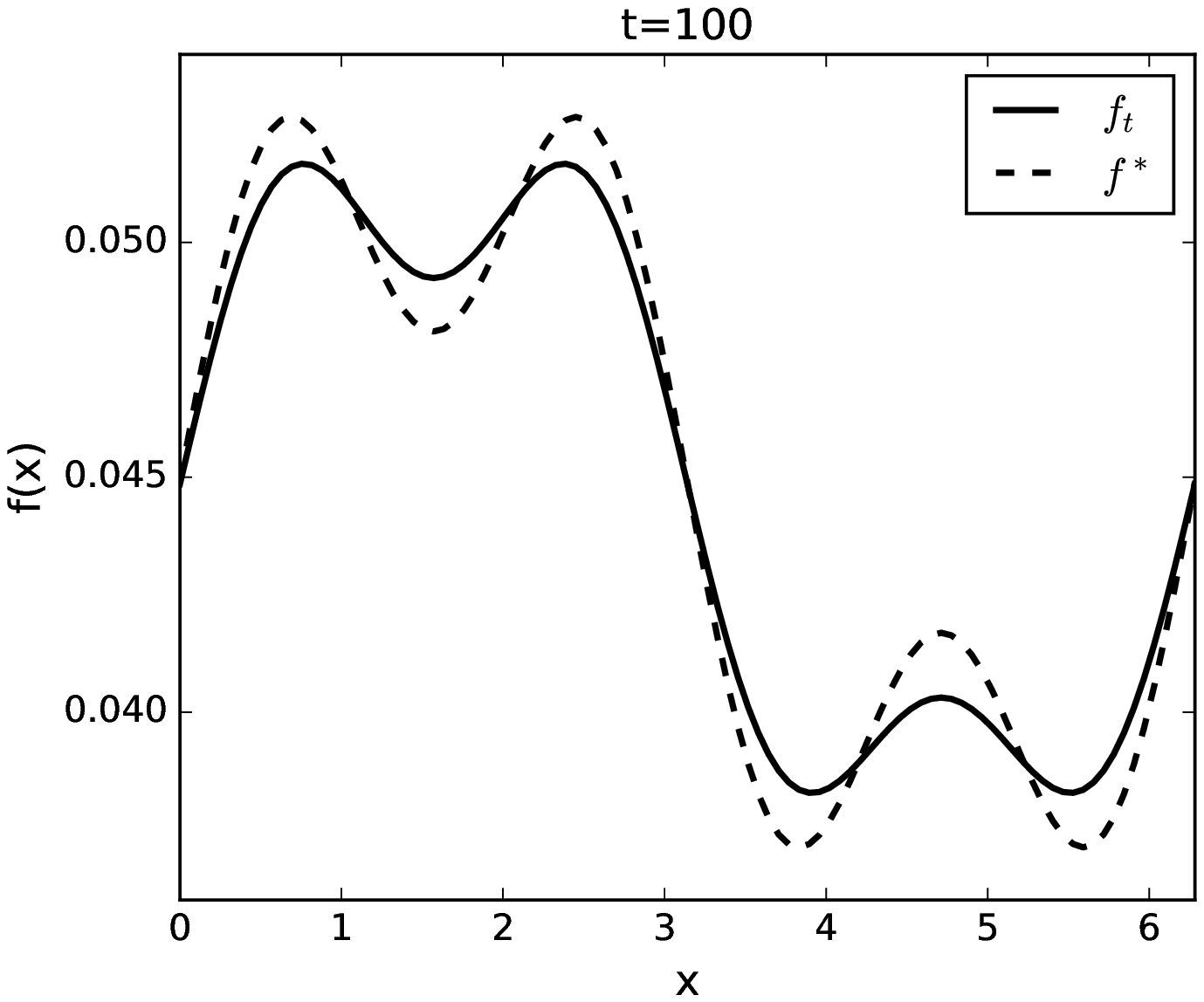}
    \includegraphics[width=0.35\textwidth]{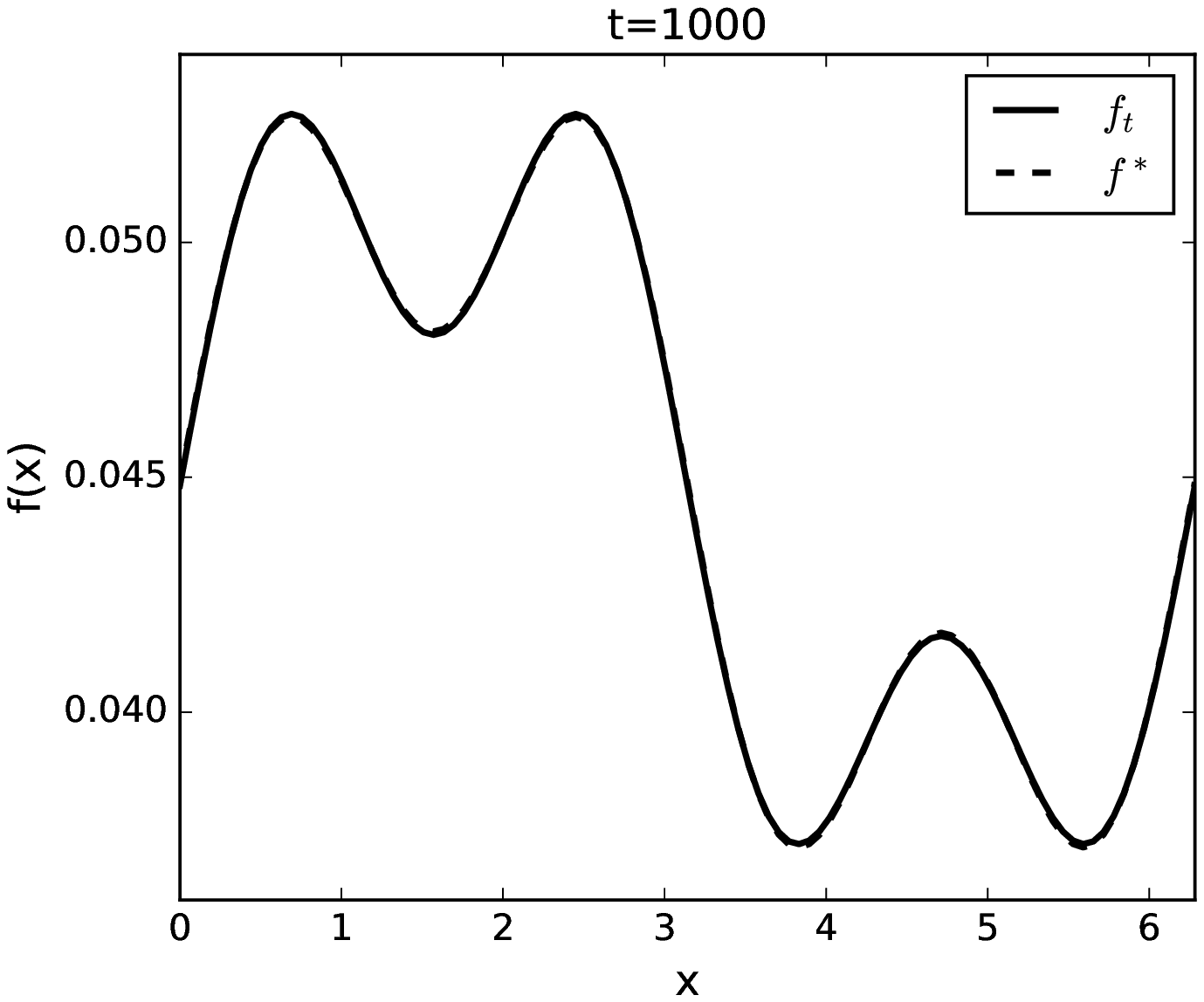}
    \vspace*{-2mm}
    \caption{\small The example that demonstrates the frequency principle. The four plots correspond to the function $f_t$ at $t=0,\ 10,\ 100,\ 1000$, compared to the target function.}
    \label{fig: 1d_solution}
\end{figure}

However, one should not expect this simple picture to literally hold in the general case.
In Figure~\ref{fig: 1d_solution_2}, we show the results for an example with $h=0.2$ and 
\begin{equation}
    \rho^*(w) = \frac{1}{2\pi}(1+0.5\times \sin(w)+0.5\times \sin(5w)).
\end{equation}

In this case, for $k\leq 2/h=5$, the eigenvalue increases with $k$. Thus, we see that the high-frequency part ($k=5$) converges faster than the low frequency part ($k=1$).
This is the consequence of the interplay between the frequency components
in the target function and the spectrum of $K$. 
When there is a concentration of energy in the intermediate range of the spectrum
for the target function, one should expect the scenario shown in Figure~\ref{fig: 1d_solution_2}
to happen.

\begin{figure}
    \centering
    \includegraphics[width=0.35\textwidth]{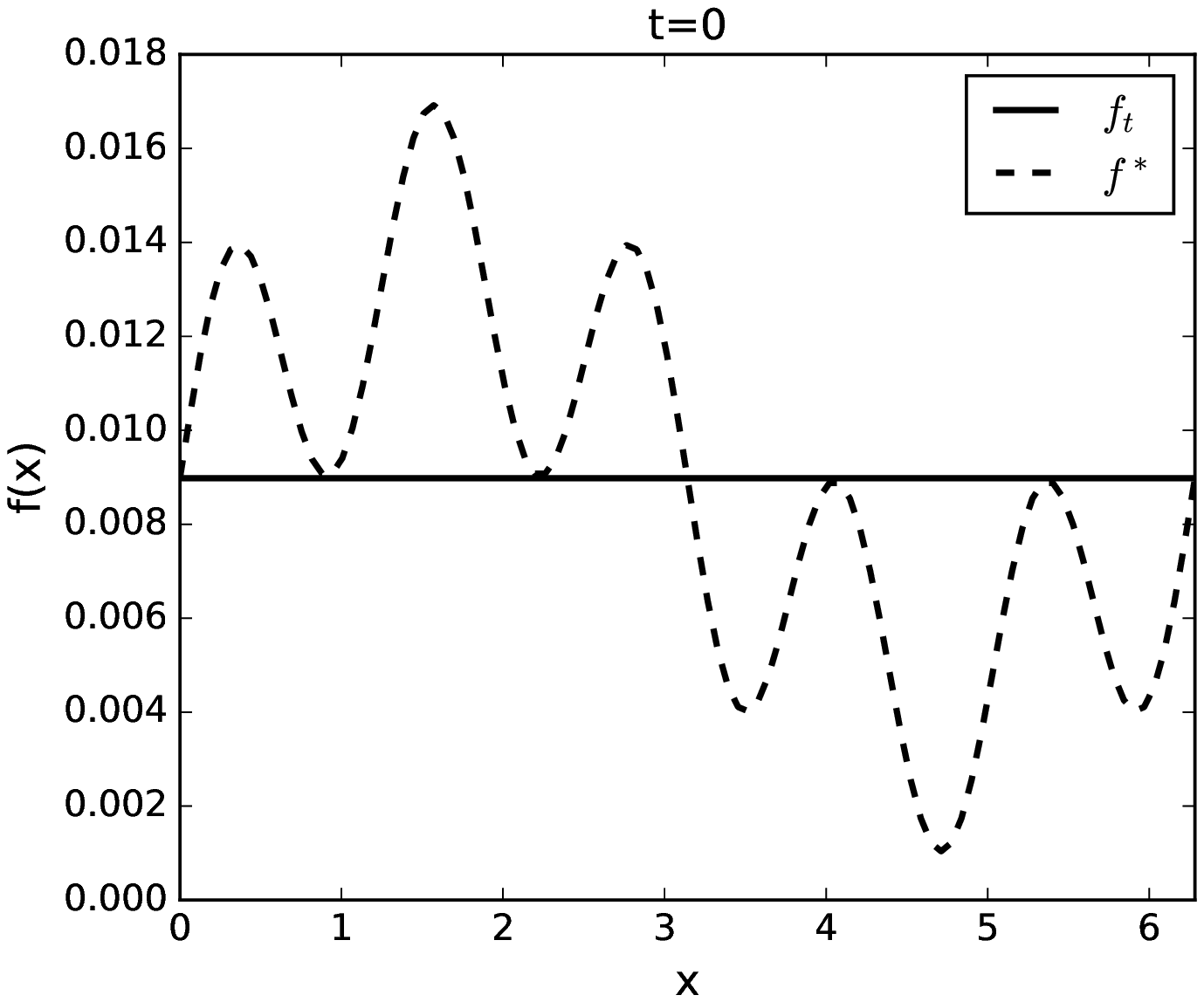}
    \includegraphics[width=0.35\textwidth]{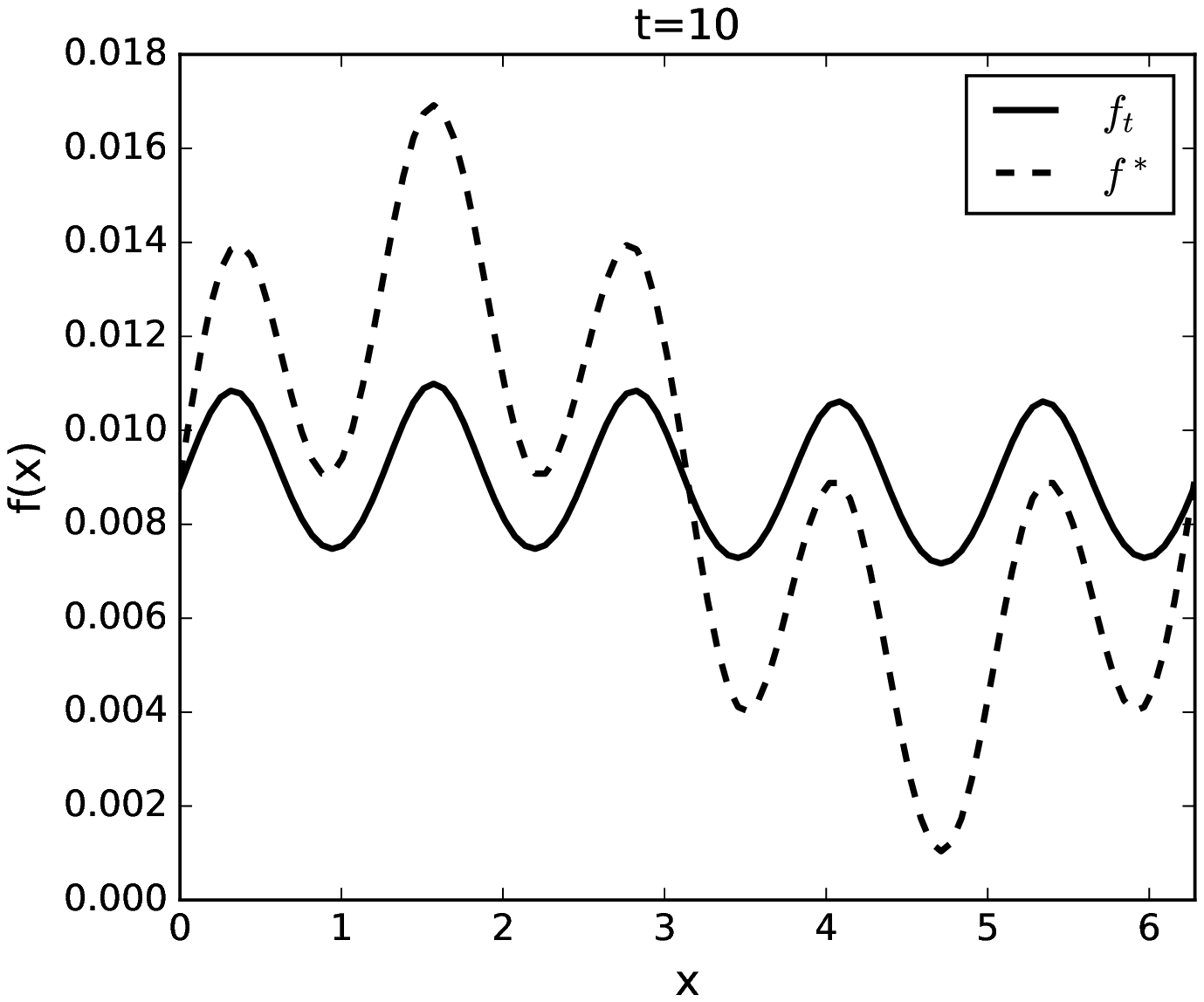}

    \includegraphics[width=0.35\textwidth]{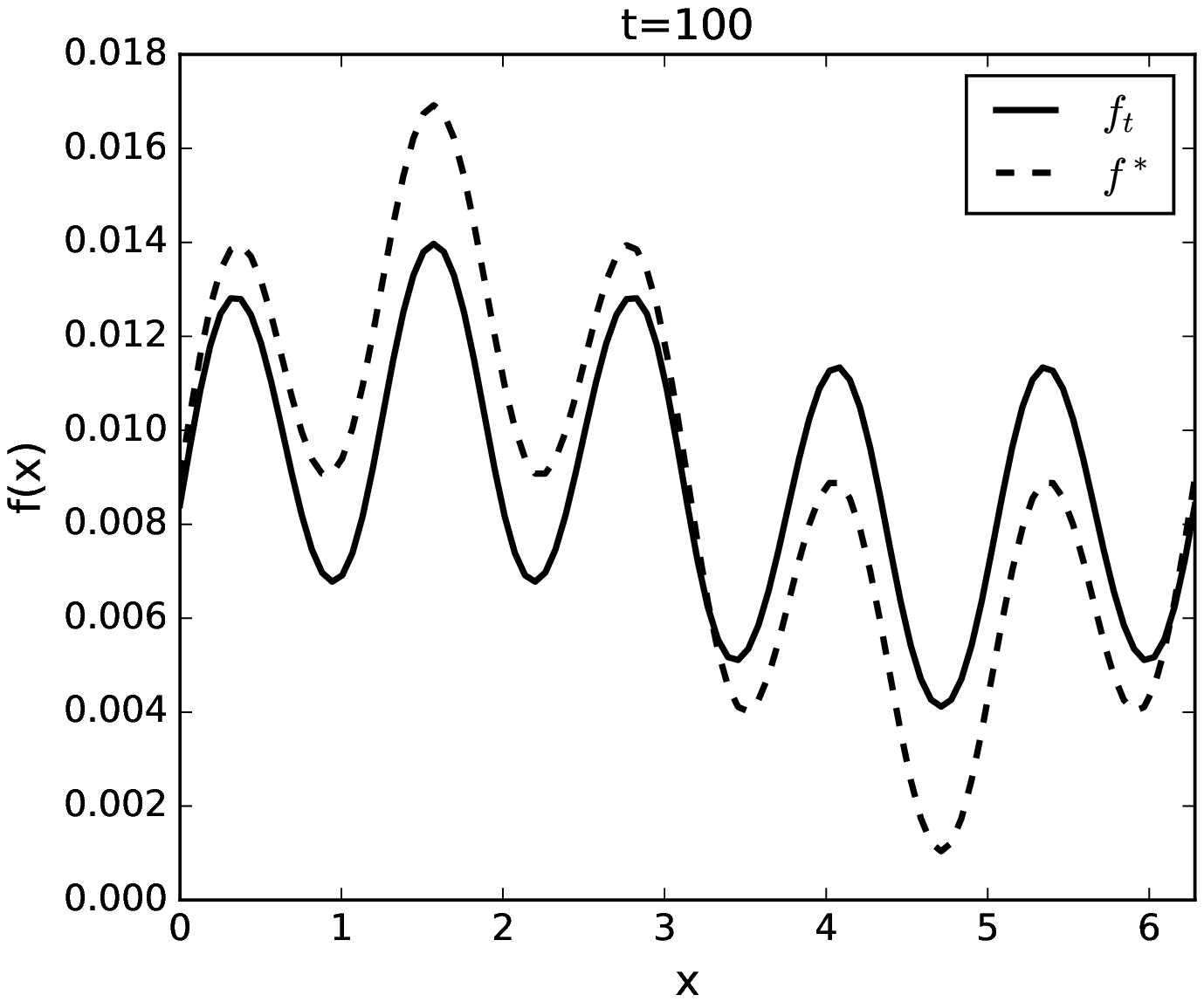}
    \includegraphics[width=0.35\textwidth]{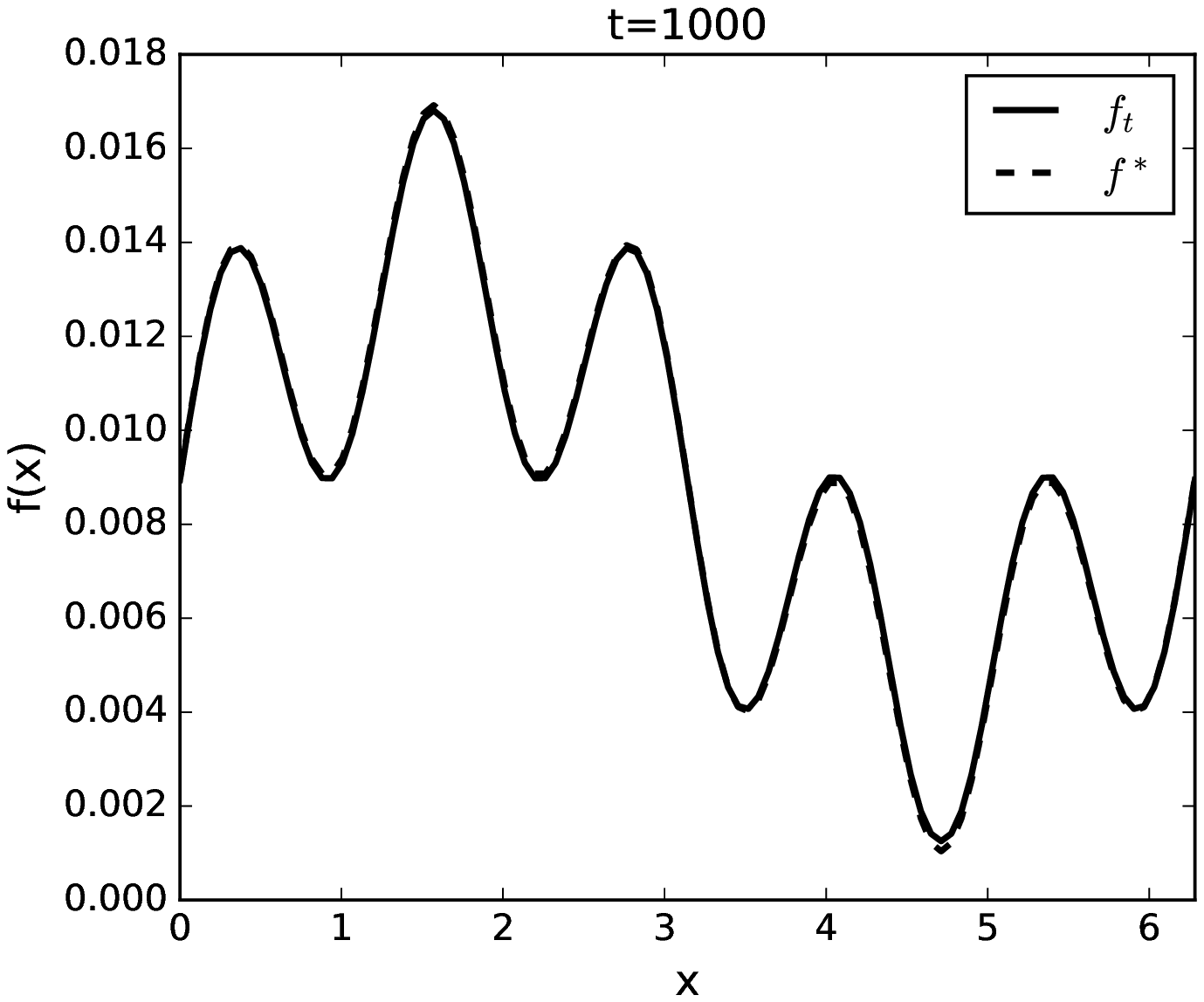}
    \vspace*{-3mm}
    \caption{\small 
The example that demonstrates when the frequency principle does not hold. The four plots correspond to the function $f_t$ at $t=0,\ 10,\ 100,\ 1000$, compared to the target function.}
    \label{fig: 1d_solution_2}
\end{figure}

\section{Discussions}

The continuous viewpoint presented here
offers a more abstract way of thinking about machine learning.
Instead of thinking about features and neurons, one focuses on the representation of
functions, the calculus of variation problem, and the continuous gradient flow. 
Features and neurons arise as objects used in special discretizations of these 
continuous problems. 

We learn at least two things from this thought process.
On one hand we can discuss machine learning without appealing
to the idea of neurons, and indeed there are plenty of algorithms and models besides
the neural network models. 
On the other hand, we also see  why neural networks, both shallow and deep (ResNet),
are inevitable choices: They are the simplest particle method discretization of the
simplest continuous gradient flow models 
(for the integral transform-based and flow-based representations respectively).

One main theme in classical numerical analysis is to 
come up with design principles for better models and better algorithms.
In that spirit, one can suggest the following set of principles for the continuous approach:
\begin{enumerate}
\item The target functions should be represented as expectations in various forms.
\item The risk functionals should be nice functionals. Even if not convex, they should share
many features of convex functionals. A good thing is that if we start from as continuous mode, 
it is likely that the discretized model will not be plagued by local minima that results of discrete effects.
\item The different gradient flows are nice flows in the sense that the relevant norms should behave well under
the flow. Here the ``relevant norm'' means  the norm associated with the particular representation (e.g. Barron norm for the 
integral transform-based representation).
\item The numerical discretization of the flow should be stable over long time intervals.
\end{enumerate}
We suspect that if one follows this set of design principles,  the resulting models and algorithms
will behave in a rather robust fashion, in contrast to current machine learning models which tend to
depend sensitively on the choice of hyper-parameters.

Some of the subtleties in current machine learning algorithms can already be appreciated just by looking at
things from a continuous viewpoint.
For example, very deep fully connected networks should cause problems since they 
do not have nice continuum limits \cite{hanin2018neural}.

\subsection*{Acknowledgement}
The work presented here is supported in part by a gift to Princeton University from iFlytek
and the ONR grant N00014-13-1-0338.
We are grateful to Jianfeng Lu, Stephan Wojtowytsch, Lexing Ying and Shuhai Zhao for helpful
discussions.

{\small 
\bibliographystyle{plain}
\bibliography{dl_ref}
}

\end{document}